\documentclass[11pt]{article}
\usepackage{latexsym}
\usepackage[english]{babel}

\usepackage[letterpaper,top=2cm,bottom=2cm,left=2cm,right=2cm]{geometry}
\usepackage{microtype}
\usepackage{mathtools}
\usepackage{mathrsfs}
\usepackage{algorithm}
\usepackage[noend]{algpseudocode}
\usepackage{amssymb}
\usepackage{graphicx}
\usepackage[colorlinks=true, allcolors=blue]{hyperref}
\usepackage{subcaption}
\usepackage{bm}
\usepackage{amsmath}
\numberwithin{equation}{section}
\usepackage{booktabs}
\usepackage{tabularx}
\usepackage{multirow}
\usepackage{enumerate}
\usepackage{amsthm}

\usepackage{fancyhdr}
\usepackage{amssymb}
\newtheorem{theorem}{Theorem}[section]
\newtheorem{lemma}[theorem]{Lemma}
\newtheorem{proposition}[theorem]{Proposition}
\newtheorem{corollary}[theorem]{Corollary}

\newtheorem{assumption}[theorem]{Assumption}
\newtheorem{remark}[theorem]{Remark}
\allowdisplaybreaks

\def\e{\mathbb{E}}

\def\rbb{\mathbb{R}}

\def\p{\operatorname{prox}}
\newcommand{\argmin}{\operatorname{\arg\min}}

\title{\bf \Large A Single-Loop Stochastic Proximal Quasi-Newton Method for Large-Scale Nonsmooth Convex Optimization}
\author{Yongcun Song\thanks{Division of Mathematical Sciences, School of Physical and Mathematical Sciences, Nanyang Technological University, 21 Nanyang Link, 637371, Singapore (\texttt{yongcun.song@ntu.edu.sg}).}
\and Zimeng Wang\thanks{Department of Mathematics and Mathematical Statistics, Umeå University, 90187 Umeå, Sweden (\texttt{zimeng.wang@umu.se}).}
\and Xiaoming Yuan\thanks{Department of Mathematics, The University of Hong Kong, Hong Kong, China (\texttt{xmyuan@hku.hk}).}
\and Hangrui Yue\thanks{School of Mathematical Sciences, Nankai University, Tianjin, 300071, China (\texttt{yuehangrui@gmail.com}).}}
\date{}

\begin{document}
\maketitle

\begin{abstract}
  We propose a new stochastic proximal quasi-Newton method for minimizing the sum of two convex functions in the particular context that one of the functions is the average of a large number of smooth functions and the other one is nonsmooth. The new method integrates a simple single-loop SVRG (L-SVRG) technique for sampling the gradient and a stochastic limited-memory BFGS (L-BFGS) scheme for approximating the Hessian of the smooth function components.
  The globally linear convergence rate of the new method is proved under mild assumptions. It is also shown that the new method covers a proximal variant of the L-SVRG as a special case, and it allows for various generalization through the integration with other variance reduction methods. For example, the L-SVRG can be replaced with the SAGA or SEGA in the proposed new method and thus other new stochastic proximal quasi-Newton methods with rigorously guaranteed convergence can be proposed accordingly. Moreover, we meticulously analyze the resulting nonsmooth subproblem at each iteration and leverage a compact representation of the L-BFGS matrix with the storage of some auxiliary matrices. As a result, we propose a very efficient and easily implementable semismooth Newton solver for solving the involved subproblems, whose arithmetic operations per iteration are merely order of $O(d)$, where $d$ denotes the dimensionality of the problem. With this efficient inner solver, the new method performs well and its numerical efficiency is validated through extensive experiments on a regularized logistic regression problem.

         \medskip


    \medskip

\end{abstract}

\section{Introduction}
We consider the following nonsmooth convex optimization problem:
\begin{equation} \label{eq: main prob}
	\min _{x \in \mathbb{R}^d} F(x) :=  \underbrace{\frac{1}{n} \sum_{i=1}^n f_i(x)}_{f(x)}+h(x),
\end{equation}
in which $f: \rbb^d \rightarrow \rbb$ is the average of $n$ function components and each component $f_i : \rbb^d \rightarrow \rbb$ is convex and differentiable (hence $f$), and the function $h: \rbb^d \rightarrow \rbb\cup\{+\infty\}$ is convex, proper, and lower semicontinuous, but possibly nonsmooth. Problem (\ref{eq: main prob}) with large $n$ is prevalent within the machine learning community, known as regularized empirical risk minimization (see, e.g., \cite{hastie2009elements}), where each component $f_i$ represents the loss associated with the $i$-th data sample, while $h$ is a manually incorporated regularization function aimed at improving model stability or promoting sparsity in the model parameters.
Typical machine learning models that fall into the frame of problem \eqref{eq: main prob} include and are not limited to LASSO \cite{tibshirani1996regression}, support vector machine \cite{cortes1995support}, and regularized logistic regression \cite{berkson1944application}.

\subsection{Deterministic methods}

For the generic $f$, it is common to solve problem (\ref{eq: main prob}) via considering its approximation scheme
\begin{equation} \label{eq: ista}
	x_{k+1}=\underset{x\in \rbb^d}{\argmin}~\left\{f\left(x_{k}\right)+ \nabla f\left(x_{k}\right)^\top( x-x_{k})+\frac{1}{2 \eta_k}\|x-x_k\|^2 + h(x)\right\},
\end{equation}
with $\eta_k>0$. That is, the smooth function $f$ is approximated by a simpler quadratic function iteratively. With the notation of proximal operator $\p_{\eta,\;h}(\cdot)$ (see \eqref{eq: prox-eta} for the definition), (\ref{eq: ista}) can be equivalently written as
\begin{equation}\label{eq: proximal}
x_{k+1}=\p_{\eta_k,\;h}(x_{k}-\eta_k \nabla f(x_k)).
\end{equation}
The iterative scheme (\ref{eq: proximal}) is also called the proximal gradient (PG) method \cite{lions1979splitting, passty1979ergodic}. Depending on $h$, (\ref{eq: proximal}) could be easy enough to have a closed-form solution, and a set of classic algorithms can be rendered from (\ref{eq: proximal}) with different specifications of $h$. Examples include the projected gradient method \cite{levitin1966constrained} when $h$ is the indicator function of a convex and closed set in $\mathbb{R}^d$, and the iterative shrinkage-thresholding algorithm (ISTA), as well as its faster version (FISTA) in \cite{beck2009fast} when $h=\lambda\|\cdot\|_1$ for some $\lambda>0$.
It is analyzed in, e.g., \cite{beck2009fast, nesterov2013gradient}, that the PG methods exhibit sublinear and linear convergence rates when the smooth function $f$ is convex and strongly convex, respectively. Nevertheless, the approximated model (\ref{eq: ista}) only involves the gradient information of $f$, while no information of the Hessian of $f$ is considered. Hence, implementation of the resulting iterations is relatively easier, yet approximation to the original problem (\ref{eq: main prob}) is less accurate and thus the theoretical convergence rate of such an algorithm is of lower order.

To achieve higher-order convergence rates and further acceleration, it is necessary to incorporate the Hessian of $f$ or its approximation into the approximation of the original problem (\ref{eq: main prob}). Some proximal Newton-type methods were thus proposed in the literature, to mention a few, see \cite{becker2012quasi, becker2019quasi, byrd2016family, byrd2016inexact, lee2014proximal, mordukhovich2021generalized, mordukhovich2023globally, stella2017forward}. These proximal Newton-type methods have a common feature that the original problem (\ref{eq: main prob}) is approximated via the scheme
\begin{equation} \label{eq: prox_newton}
	x_{k+1}=\underset{x\in \rbb^d}{\argmin}~\left\{f\left(x_{k}\right)+\nabla f\left(x_{k}\right)^\top( x-x_{k})+\frac{1}{2 \eta_k}(x-x_k)^\top B_k (x-x_k) + h(x)\right\},
\end{equation}
where $\eta_k>0$ can be regarded as the step size and $B_k\in\rbb^{d\times d}$ is the Hessian matrix $\nabla^2f(x_k)$ or its approximation. In particular, if $B_k$ is set as $\nabla^2 f(x_k)$ or constructed by a quasi-Newton strategy, we obtain the proximal Newton or proximal quasi-Newton method \cite{lee2014proximal} accordingly. Note that the PG method (\ref{eq: proximal}) is recovered from \eqref{eq: prox_newton} when $B_k$ is the identity matrix. Of course, compared with (\ref{eq: ista}), solving (\ref{eq: prox_newton}) is generally more difficult due to the presence of the general $B_k$. Hence, the computation per iteration of such a resulting algorithm is more time-consuming. Nevertheless, the curvature information of $f$ can be well captured in (\ref{eq: prox_newton}) and the original problem (\ref{eq: main prob}) is approximated more accurately. Hence, the theoretical convergence rate of such an algorithm based on (\ref{eq: prox_newton}) is of higher order. Indeed, it has been shown in \cite{lee2014proximal} that, under the same assumptions as those for the PG, the proximal Newton and the proximal quasi-Newton methods with unit length step sizes can achieve locally quadratic and superlinear convergence, respectively.

For proximal quasi-Newton methods, it is crucial to construct the Hessian approximation $B_k$ efficiently. In this regard, the BFGS \cite{fletcher2000practical} and the limited-memory BFGS (L-BFGS) \cite{liu1989limited} are widely used in the literature.
In the BFGS, the inverse Hessian approximation $H_k := B_k^{-1}$ is necessary to be stored to update $H_{k+1}$, which requires $O(d^2)$ memory.
By contrast, the L-BFGS only needs to store $m$ correction pairs
$$
\{(s_i, y_i)\}_{i=k-m+1}^{k}\; \hbox{with} \; s_i=x_{i+1}-x_i \;\hbox{and} \; y_i=\nabla f\left(x_{i+1}\right)-\nabla f\left(x_i\right)
$$
to construct $H_{k+1}$ and hence reduces the storage costs to $O(md)$, where $m>0$ is the memory size. Consequently, the L-BFGS is preferred for practical implementations, particularly when $d$ is large.

These deterministic methods are efficient when the full gradient $\nabla f:=\frac{1}{n}\sum_{i=1}^n\nabla f_i$ is accessible. They find extensive applications across various fields, including barrier representation of feasible set \cite{nesterov1994interior}, LASSO \cite{tibshirani1996regression}, and many others.
However, in modern machine learning tasks, the number of components $n$ involved in \eqref{eq: main prob} is usually very large, which makes the full gradient $\nabla f$ very expensive to compute and thus deteriorates the numerical efficiency.

\subsection{Stochastic methods}
In the context of machine learning, a common idea for applying gradient-based methods on large-scale optimization problems is to replace a full gradient with a certain stochastic approximation, such as a single or a small batch of gradient components. This results in the class of stochastic algorithms, which can be traced back to the seminal work \cite{robbins1951stochastic}.
The predominant methodology within this class advocates the stochastic gradient descent (SGD) method \cite{robbins1951stochastic} and its variants \cite{duchi2011adaptive, kingma2014adam,loizou2020momentum,qian1999momentum,zeiler2012adadelta}, which have demonstrated significant success across a wide range of tasks in machine learning \cite{ bottou2007tradeoffs,bottou1991stochastic, lecun1998gradient}. The per-iteration cost of SGD and its variants is relatively low but their worst-case convergence rates are only sublinear even for strongly convex functions, as the step sizes are necessary to vanish to ensure convergence.
To address this issue, various approaches have been proposed in the literature, including variance reduction (VR) methods \cite{bottou2018optimization, defazio2014saga, gower2020variance, hanzely2018sega, johnson2013accelerating, roux2012stochastic, shalev2013stochastic, xiao2014proximal} and stochastic Newton-type methods \cite{byrd2016stochastic, guo2023overview, kasai2018riemannian, lucchi2015variance, moritz2016linearly,  schraudolph2007stochastic, zhao2017stochastic}.

Different from the SGD that uses one or more gradient components directly as the full gradient approximation, VR methods use them to update it, such that the variance of the gradient approximation vanishes as iterations progress, and hence often exhibit faster convergence both in theory and practice.
Typical VR methods include SAG \cite{roux2012stochastic}, SAGA \cite{defazio2014saga}, SDCA \cite{shalev2013stochastic}, SVRG \cite{johnson2013accelerating, xiao2014proximal}, SEGA \cite{hanzely2018sega} and so on. These VR methods admit constant step sizes and thus can achieve linear convergence rates for strongly convex objectives.
Among these VR methods, the SVRG stands out as a popular choice due to its low storage cost and satisfactory performance in various tasks \cite{johnson2013accelerating, li2020accelerated}.
The SVRG method is mainly featured by a double-loop structure: in the outer loop, a full gradient is computed at a reference point, and this full gradient is then used in the inner loop to modify the gradient approximations.

Stochastic Newton-type methods incorporate Hessian information and hence often exhibit faster convergence than stochastic first-order methods. Representative stochastic Newton-type methods include the online BFGS (oBFGS), the online L-BFGS (oL-BFGS) \cite{schraudolph2007stochastic}, the stochastic quasi-Newton (SQN) method \cite{byrd2016stochastic} and the stochastic L-BFGS method \cite{moritz2016linearly}.
The oBFGS and the oL-BFGS represent pioneering approaches that generalize the BFGS and the L-BFGS to stochastic settings.
The SQN method leverages a modified L-BFGS scheme that applies a Hessian-vector product to update the Hessian approximation. While the worst-case convergence rate of SQN is sublinear, practical performance of SQN surpasses that of SGD significantly, indicating the substantial advantage offered by incorporating Hessian information. The stochastic L-BFGS \cite{moritz2016linearly} builds upon the SQN and the SVRG, and is the first stochastic quasi-Newton algorithm achieving globally linear convergence for strongly convex and smooth objectives. Some methods that combine SQN with the SVRG can be referred to \cite{kasai2018riemannian, lucchi2015variance, zhao2017stochastic}. 

For solving  \eqref{eq: main prob}, some stochastic PG methods have been studied in \cite{bertsekas2011incremental, duchi2009efficient, langford2009sparse}.
Additionally, some methods combining the VR techniques and stochastic Newton-type methods have been proposed in the literature. For instance, a general framework designed in \cite{yang2021stochastic} performs an additional stochastic proximal gradient step at each iteration after a Newton-type update and allows for integration with VR methods.
The subsampled Newton method with cubic regularization proposed in \cite{zhang2022adaptive} assumes access to the true Hessian of each component function, and aggregates the stochastic gradient and Hessian approximation using a subset of components according to the SVRG scheme.
An inexact subsampled proximal Newton method combined with the SVRG is developed in \cite{wang2019utilizing}, and its convergence results are established for a special case of \eqref{eq: main prob}, where each component $f_i$ is a loss of a linear predictor.
In \cite{luo2016proximal}, a stochastic proximal quasi-Newton method that combines the SQN method with the SVRG technique is introduced and analyzed. Specifically, at the $k$-th iteration of this method, the new iterate $x_{k+1}$ is generated via \eqref{eq: prox_newton} or its approximation, with $\nabla f(x_k)$ replaced by a stochastic gradient generated via the SVRG.

Note that the aforementioned stochastic methods that incorporate VR techniques simply adopt the SVRG scheme, see \cite{kasai2018riemannian, lucchi2015variance, luo2016proximal, moritz2016linearly, wang2019utilizing, zhang2022adaptive, zhao2017stochastic}.
Despite achieving fast theoretical convergence, these methods are limited by the inherent drawbacks of the SVRG resulting from its double-loop structure.
For instance, existing convergence results of the SVRG are established for the reference point, updating which can be computationally expensive since a full gradient computation is required.
Additionally, as pointed out in \cite{kovalev2020don}, the practical performance of the SVRG is heavily influenced by the number of inner iterations, but there lacks theoretical guidance on selecting its value optimally. Moreover, empirical observations have indicated that a simplified implementation of the SVRG often exhibits superior performance in practice but without convergence guarantee, resulting in a discrepancy between theory and practice (see Section \ref{se:construct vk} for the details).

\subsection{Our contributions}

Existing stochastic quasi-Newton methods for solving the nonsmooth problem \eqref{eq: main prob} suffer from either slow sublinear convergence or practical issues arising from the double-loop structure of the SVRG.
To overcome these limitations, we combine a stochastic L-BFGS scheme \cite{byrd2016stochastic} based on the Hessian-vector product with the recently introduced loopless SVRG (L-SVRG) \cite{kovalev2020don} and hence propose a novel stochastic proximal quasi-Newton method for solving \eqref{eq: main prob}. Moreover, we provide rigorous analysis demonstrating its rapid globally linear convergence and present a highly efficient way for numerical implementation.

To adapt the L-BFGS for stochastic optimization, an intuitive way is to generate the correction vector $y$ by differencing stochastic gradients based on small samples. Unfortunately, as pointed out in \cite{byrd2016stochastic}, this approach may lead to a biased estimate of the Hessian approximation.
To address this issue, we follow \cite{byrd2016stochastic} and decouple the calculations of the stochastic gradients and Hessian approximations using Hessian-vector products, which also provide the flexibility for periodic updates of the Hessian approximations.
Moreover, to overcome the limitations brought by the double-loop structure of the SVRG, the L-SVRG eliminates the inner loop by incorporating a probability-based update of the reference point at each iteration.
It was shown in \cite{kovalev2020don} that, compared with the SVRG, the L-SVRG achieves a linear convergence rate for strongly convex objectives without increasing storage cost while exhibiting a simplified structure.
Furthermore, empirical results demonstrate that the L-SVRG outperforms the SVRG in practical applications.

Combining the ideas of the stochastic L-BFGS and the L-SVRG, our proposed method proceeds by performing the following iterative scheme
\begin{equation} \label{eq: our_alg}
	x_{k+1}=\underset{x\in \rbb^d}{\argmin}~\left\{f\left(x_{k}\right)+ v_k^\top( x-x_{k})+\frac{1}{2 \eta_k}(x-x_k)^\top B_k (x-x_k) + h(x)\right\},
\end{equation}
where $v_k\in \rbb^d$ denotes the stochastic gradient obtained by the L-SVRG, and $B_k\in \rbb^{d\times d}$ is constructed via $m$ correction pairs generated by the approach outlined in  \cite{byrd2016stochastic}.
The proposed method surpasses SVRG-based proximal quasi-Newton methods by featuring a unique single-loop structure to update $v_k$, which makes it easier and cheaper to implement.
In addition to the simple structure, our method demonstrates a rapid globally linear convergence rate under the same assumptions as those for the existing stochastic proximal quasi-Newton methods in the literature. To the best of our knowledge, our method seems to be the first stochastic proximal quasi-Newton method that incorporates a single-loop stochastic gradient updating scheme while preserving the desirable property of linear convergence.
Moreover, as a special case of our method, we obtain a proximal extension of the L-SVRG for addressing \eqref{eq: main prob}.
We also explore some generalizations of our method, where the stochastic gradients are generated by other VR methods like the SAGA and the SEGA. Notably, linear convergence results for these variants can be established by extending the mathematical paradigm used in analyzing the original method (see Section \ref{se:generalization} for more details).

Like other proximal quasi-Newton methods, the practical efficiency of our method heavily relies on the rapid solution of the subproblem \eqref{eq: our_alg}. If a first-order algorithm is employed for this purpose, one may see slow convergence and struggle in pursuing highly accurate solutions.
To enhance the practical applicability of our method, we analyze the nonsmooth subproblem \eqref{eq: our_alg} meticulously and propose a new inner solver for this subproblem that can obtain highly accurate solutions very fast. Specifically, we first transform the nonsmooth subproblem into an equivalent smooth dual formulation. Then, we propose a Semismooth Newton (SSN) method \cite{qi1993nonsmooth} along with a line search scheme to efficiently solve the dual problem by carefully exploring its specific mathematical structure.
To further reduce the computational cost, we develop an efficient numerical implementation for the proposed SSN method using a compact representation of the L-BFGS matrix and introducing auxiliary matrices that can be updated in a highly efficient manner. Generally, such an implementation requires only $O(\iota md)$  multiplications and $O(\iota m^2d)$ additions, along with $O(\iota d)$ simple non-linear operations (e.g., projections). Here, $\iota>0$ (typically a single-digit number) denotes the number of SSN iterations.
It is noteworthy that our SSN method can effectively address the nonsmooth problem \eqref{eq: our_alg} with general $v_k$ and positive definite matrix $B_k$, which is commonly incorporated as an inner subproblem in various proximal Newton-type methods, such as the proximal quasi-Newton method \cite{becker2012quasi, becker2019quasi, byrd2016inexact, lee2014proximal}, the stochastic proximal Newton-type method integrated with the SVRG \cite{luo2016proximal, wang2019utilizing} and the proximal subsampled Newton method \cite{liu2017inexact}. Consequently, the proposed SSN solver can be readily integrated into these proximal Newton-type methods as a subroutine, leading to immediate enhancements in their numerical performance.

\subsection{Organization}
The rest of the paper is organized as follows. Section \ref{sec: pre} provides some notations and preliminary results that are used throughout this paper. Then the proposed algorithm is introduced in Section \ref{sec: design}, followed by the convergence analyses in Section \ref{sec: analysis}.  Section \ref{sec: ssn} presents an SSN approach for efficiently solving the subproblem \eqref{eq: our_alg}. Section \ref{sec: evaluation} is devoted to numerical experiments and Section \ref{sec: conclusion} concludes the paper.

\section{Preliminaries}
\label{sec: pre}
In this section, we summarize some notations and preliminary results that are used throughout the paper. First,
we denote by $I_d$ the $d\times d$ identity matrix and omit the subscript $d$ when it is clear from the context.
For a vector $v\in \rbb^d$, we denote by $\|v\|$ its Euclidean norm and by $\|v\|_B:=\sqrt{v^\top B v}$ the $B$-norm of $v$, with $B\in \rbb^{d\times d}$ being a positive definite matrix.
For symmetric matrices $A, B\in\rbb^{d\times d}$, we write $A \preceq B$ if the matrix $B-A$ is positive semidefinite. 

Let ($\Omega, \mathcal{F}, \mathrm{P})$ be the underlying probability space. We denote by $\left\{x_k\right\}_{k \geq 0}\subset\rbb^d$ the sequence of iterates generated by a stochastic algorithm, driven by a sequence of independent noise variables $\left\{\xi_k\right\}_{k \geq 1}$. Let $\left\{\mathcal{F}_k\right\}_{k \geq 0}$ be the natural filtration of this process, i.e., $\mathcal{F}_k= \sigma\left(x_0, \xi_1, \xi_2, \ldots, \xi_k\right)$ represents the $\sigma$-algebra generated by $x_0$ and $\left\{\xi_i\right\}_{i=1}^k$. We denote the conditional expectation given $\mathcal{F}_k$ by $\e_k[\cdot]:=\e\left[\cdot \mid \mathcal{F}_k\right]$, which represents the expectation given all information available at step $k$.

For a set $\mathcal{S}\subset [n]:=\{1, 2, \dots, n\}$ with cardinality $|\mathcal{S}|$ and functions $\{f_i\}_{i=1}^n$, we define
\begin{equation*}
	f_{\mathcal{S}}(x)=\frac{1}{|\mathcal{S}|}\sum_{i \in \mathcal{S}} f_i(x),~\forall x\in\mathbb{R}^d.
\end{equation*}
If each $f_i$ is smooth enough, we define $\nabla f_{\mathcal{S}}(x)$ and $\nabla^2 f_{\mathcal{S}}(x)$ in the same way.

A function $f: \rbb^d \rightarrow \rbb$ is said to be $L$-smooth if it is continuously differentiable and its gradient $\nabla f$ is $L$-Lipschitz continuous.
If $f$ is $L$-smooth and convex, then we have
\begin{equation}\label{eq: cvx-L}
	f(y) \ge f(x) + \nabla f(x)^\top(y-x) +\frac{1}{2L}\|\nabla f(y)-\nabla f(x)\|^2,~\forall x,y \in \mathbb{R}^d.
\end{equation}
One can refer to \cite{nesterov2003introductory} for a proof of \eqref{eq: cvx-L}. 

Using \eqref{eq: cvx-L}, we have the following result that is widely used in the convergence analysis of proximal-type algorithms.
\begin{lemma}[{\cite[Lemma 1]{xiao2014proximal}}] \label{lem: pre_bv}
	Consider $F$ as defined in \eqref{eq: main prob}.
	Suppose for each $i\in [n]$, $f_i$ is convex and $L$-smooth. If $x^*\in \argmin_{x\in\rbb^d} F(x)$, then we have
	\begin{equation*}
		\frac{1}{n}\sum_{i=1}^{n}\|\nabla f_{i}(x)-\nabla f_{i}(x^{*})\|^{2}\leq2L(F(x)-F(x^*)).
	\end{equation*}
\end{lemma}

Let $F:\rbb^d \rightarrow \rbb$ be a convex function that attains a global minimizer, we denote $F^* := \min_{x\in\rbb^d} F(x)$. If $F$ is moreover strongly convex, then it admits a unique minimizer $x^* := \argmin_{x\in\rbb^d} F(x)$.

For a proper, lower semicontinuous, and convex function $h: \rbb^d \rightarrow \rbb \cup \{+\infty\}$, we denote by $\partial h(x)$ the set of all subgradients of $h$ at $x\in \rbb^d$, i.e.,
\begin{equation*}
    \partial h(x) := \{\gamma \in \rbb^d | h(y) - h(x) - \langle \gamma, y-x \rangle \geq 0, \forall y\in \rbb^d\}.
\end{equation*}
Its proximal operator with parameter $\eta > 0$ is defined as
\begin{equation}\label{eq: prox-eta}
	\p_{\eta,\;h}(x)=\underset{z\in \rbb^d}{\argmin}~\left\{h(z)+\frac{1}{2\eta}\|z-x\|^2\right\},~\forall x\in \rbb^d,
\end{equation}
and the Moreau envelope of $h$ with parameter $\eta > 0$ is given by
\begin{equation*}
	\mathcal{M}_{\eta,\;h}(x)= \min_{z\in \rbb^d}~h(z) +\frac{1}{2\eta}\|z-x\|^2, ~\forall x\in \rbb^d.
\end{equation*}
It has been shown in \cite{rockafellar1997convex} that $\mathcal{M}_{\eta,\;h}$ is differentiable everywhere and its gradient is given by
\begin{equation} \label{eq: diff_M}
	\nabla \mathcal{M}_{\eta,\;h}(x)=\frac{1}{\eta}\left(x-\p_{\eta,\;h}(x)\right),~\forall x\in \rbb^d.
\end{equation}
For simplicity, we denote $\p_h(x):=\p_{1,\;h}(x)$ and $\mathcal{M}_h(x):=\mathcal{M}_{1,\;h}(x)$.

Let $B\in \rbb^{d\times d}$ be a positive definite matrix, we define the scaled proximal operator of $h$ with parameter $\eta>0$ as
\begin{equation*}
	\p_{\eta,\;h}^B(x)=\underset{z\in \rbb^d}{\argmin}~ \left\{h(z)+\frac{1}{2 \eta}\|z-x\|_B^2\right\},~\forall x\in \rbb^d.
\end{equation*}
For any $\eta>0$, the scaled proximal operator $\p_{\eta,\;h}^B(\cdot)$ is nonexpansive in the $B$-norm, that is,
\begin{equation}\label{eq:nonexpansion}
	\left\|\p_{\eta,\;h}^B(x) - \p_{\eta,\;h}^B(y)\right\|_B \leq \left\|x-y\right\|_B, ~\forall x, y\in \rbb^d.
\end{equation}

Next, we introduce an important lemma that represents an extension of \cite[Lemma 3]{xiao2014proximal} by incorporating additional Hessian information. A similar result was mentioned in \cite{luo2016proximal} but without proof. For completeness, we present a proof of the lemma.
\begin{lemma} \label{lem: key lem}
	Let $F(x)$ be defined in \eqref{eq: main prob} and suppose that $f$ is $\mu$-strongly convex and $L$-smooth.
	For any $\eta>0$, $x, v\in\rbb^d$ and positive definite matrix $B\in \rbb^{d\times d}$, let $x^{+}:=\p_{\eta,\;h}^{B}(x-\eta B^{-1}v)$.
	Define $g:=\frac{1}{\eta}(x-x^{+})$ and $ \Delta:=v-\nabla f(x)$,
	then for any $y\in \rbb^d$, it holds that
	\begin{equation*}
		F(y) \ge F(x^+)+g^\top B(y-x)+\Delta^\top(x^+-y)+\eta\|g\|_B^2-\frac{L\eta^2}{2}\|g\|^2+\frac{\mu}{2}\|y-x\|^2.
	\end{equation*}
\end{lemma}

\begin{proof}
	Let $H=B^{-1}$. It follows from the definition of $x^+$ that
	\begin{equation}\label{s1}
		x^{+}=\underset{z\in \rbb^d}{\argmin}~\left\{h(z)+\frac{1}{2\eta}\left\|z-(x-\eta Hv)\right\|_{B}^{2}\right\}.
	\end{equation}
	Then it is easy to show that the first-order optimality condition of (\ref{s1}) reads
	\begin{equation*}
		\exists\;\xi\in\partial h(x^{+}), s.t. B(x^{+}-	x+\eta H v)+\eta \xi=0,
	\end{equation*}
	which together with $g:=\frac{1}{\eta}(x-x^{+})$ implies that
	\begin{equation*}
		\xi = B g - v.
	\end{equation*}
	Since $f$ is $\mu$-strongly convex and $h$ is convex, we have
	\begin{equation} \label{eq: Ffh}
		F(y)=f(y)+h(y)\geq f(x) + \nabla f(x)^\top(y-x)+\frac{\mu}{2}\|y-x\|^2+ h(x^+)+\xi^\top (y-x^+).
	\end{equation}
	Since $f$ is $L$-smooth, we have that
	\begin{equation} \label{eq: lsmoothf}
		f(x) \ge f(x^+)-\nabla f(x)^\top(x^{+}-x)-\frac{L}{2}\|x^+-x\|^2.
	\end{equation}
	Applying (\ref{eq: lsmoothf}) on \eqref{eq: Ffh} leads to
	\begin{equation*}
		\begin{aligned}
			F(y)\geq &f(x^+)-\nabla f(x)^\top(x^{+}-x)-\frac{L}{2}\|x^+-x\|^2+\nabla f(x)^\top(y-x)+\frac{\mu}{2}\|y-x\|^2\\
      &+ h(x^+)+\xi^\top (y-x^+)\\=&F(x^+)+\nabla f(x)^\top(y-x^+)+(Bg-v)^\top(y-x^+)+\frac{\mu}{2}\|y-x\|^2 -\frac{L}{2}\|x^+-x\|^2\\=&F(x^+)+(Bg)^\top (y-x+x-x^+)+\Delta^\top(x^+-y)-\frac{L\eta^2}{2}\|g\|^2+\frac{\mu}{2}\|y-x\|^2\\=&F(x^+)+g^\top B(y-x)+\Delta^\top(x^+-y)+\eta\|g\|_B^2-\frac{L\eta^2}{2}\|g\|^2+\frac{\mu}{2}\|y-x\|^2,
		\end{aligned}
	\end{equation*}
	which completes the proof.
\end{proof}

\section{Algorithm} \label{sec: design}
In this section, we present the proposed single-loop stochastic proximal quasi-Newton method for solving problem \eqref{eq: main prob}. Typically, the iterative scheme of stochastic proximal Newton-type methods for solving \eqref{eq: main prob} can be uniformly written as
\begin{equation}\label{eq:subproblem}
	\begin{aligned}
		x_{k+1} &= \underset{x\in \rbb^d}{\argmin}~\left\{v_k ^\top\left(x-x_k\right)+\frac{1}{2 \eta_k}\left\|x-x_k\right\|_{B_k}^2+h(x)\right\}\\
		& = \p_{\eta_k,\;h}^{B_k}(x_{k}-\eta_k H_k v_k),
	\end{aligned}
\end{equation}
where $v_k\in \rbb^d$ represents a stochastic gradient, $B_k\in \rbb^{d\times d}$ serves as an approximation of the Hessian matrix $\nabla^2 f(x_k)$, and $H_k:=B_k^{-1}\in \rbb^{d\times d}$.
It is easy to see that the effectiveness of (\ref{eq:subproblem}) hinges on well-designed $v_k$ and  $B_k$ \cite{byrd2016stochastic,yang2021stochastic}.
In the rest part of this section, we shall elaborate on the construction of $v_k$ and $B_k$. In particular, we advocate combing the L-SVRG \cite{kovalev2020don} with the stochastic L-BFGS based on the Hessian-vector product \cite{byrd2016stochastic} to compute $v_k$ and $B_k$, and thus propose an efficient and easily implementable stochastic proximal quasi-Newton method for solving \eqref{eq: main prob}.

\subsection{Construction of \texorpdfstring{$v_k$}{}} \label{se:construct vk}
In general, the stochastic gradient $v_k$ serves as an unbiased estimator of $\nabla f(x_k)$ (see, e.g., \cite{defazio2014saga, duchi2011adaptive, johnson2013accelerating, kovalev2020don, nemirovski2009robust, robbins1951stochastic}).
In the design of SQN \cite{byrd2016stochastic}, $v_k$ is set as $\nabla f_{\mathcal{S}_k}(x_{k})$, where $\mathcal{S}_k\subset [n]$ is randomly selected.
However, approximating $\nabla f(x_k)$ in this way necessitates the use of a sequence of decreasing step sizes to guarantee the convergence.
This requirement leads to a sublinear convergence rate for SQN, even if $h(x)\equiv 0$ and $f$ is strongly convex.
To accelerate the convergence, a widely adopted strategy is to use VR techniques, with the SVRG method  \cite{johnson2013accelerating} being a popular and effective choice.
The SVRG method uses reference points and has a double-loop structure. To be concrete, at the $s$-th outer loop, given a reference point $\tilde{w}^s$, we set the initial iterate $x_1^s$ of the inner loop to be $\tilde{w}^s$.
Then, at the $k$-th inner loop, the new iterate $x_{k+1}^s$ is generated by $x_{k+1}^s=x_k^s-\eta v_k^s$, where $\eta>0$ is a predefined step size and $v_k^s$ is defined as
\begin{equation} \label{eq: svrg_gradient}
	v_{k}^s=\nabla f_{i_k^s}(x_{k}^s)-\nabla f_{i_k^s}(\tilde{w}^s)+\nabla f(\tilde{w}^s),
\end{equation}
where $i_k^s\in [n]$ is randomly selected at each iteration.
Upon completion of the inner loop, the new reference point $\tilde{w}^{s+1}$ is updated using the sequence $\{x_k^s\}_{k=1}^{l_s+1}$, where $l_s>0$ denotes the number of inner iterations for the $s$-th outer loop.
There are two practical options for updating $\tilde{w}^{s+1}$.
The first one is to simply set
\begin{equation}\label{update_1}
\tilde{w}^{s+1}=x_{l_s+1}^s.
\end{equation}
Alternatively, one can update $\tilde{w}^{s+1}$ as the average of the sequence $\{x_k^s\}_{k=1}^{l_s}$, i.e.,
\begin{equation}\label{update_2}
	\tilde{w}^{s+1} = \frac{1}{l_s} \sum_{i=1}^{l_s} x_i^s.
\end{equation}

It has been shown in \cite{johnson2013accelerating, xiao2014proximal} that, with the update scheme (\ref{update_2}), the SVRG and its proximal variant converge linearly. In practice,  the update (\ref{update_1}) is often preferred due to its superior empirical performance and ease of implementation.  However, convergence guarantee for the SVRG with (\ref{update_1}) is still lacking in the literature, which creates a discrepancy between mathematical theory and practice for the SVRG method.
The double-loop structure of the SVRG method also introduces additional practical challenges. For instance, as mentioned in \cite{kovalev2020don}, the SVRG is sensitive to the number of inner iterations, for which rigorous theoretical guidance has not yet been established. Additionally, the SVRG requires the reference point $\tilde{w}^s$ to be updated at each outer loop. For a new reference point $\tilde{w}^s$, a full gradient $\nabla f(\tilde{w}^s)$ is required to update $v_k$ in (\ref{eq: svrg_gradient}), which may be computationally expensive.

To overcome the aforementioned limitations of the SVRG, we advocate the L-SVRG scheme \cite{kovalev2020don}. Compared with the SVRG, the L-SVRG eliminates the inner loops, resulting in a more efficient and simplified framework for designing $v_k$. Precisely,  the L-SVRG scheme computes $v_k$ as follows:
\begin{equation} \label{eq: vr gradient}
	v_{k}=\nabla f_{i_k}(x_{k})-\nabla f_{i_k}(w_{k})+\nabla f(w_{k}),
\end{equation}
where $i_k\in [n]$ is a randomly selected index and $w_k\in \rbb^d$ is a reference point. Different from the SVRG, we update the new reference point $w_{k+1}$ as $x_k$ with a small probability $p\in (0, 1)$, while keeping it unchanged with a probability of $1-p$, i.e.,
\begin{equation}\label{eq:update_w}
	w_{k+1}=\begin{cases}x_k&\text{with probability}\; p,\\w_k&\text{with probability}\; 1-p.\end{cases}
\end{equation}
In practice, it is very common to choose a relatively small value for $p$ (e.g., $p = O(1/n)$) to ensure that $w_k$ remains unchanged for most of the iterations, resulting in substantial savings in computational cost.
Moreover, similar to the SVRG, the L-SVRG also achieves a linear convergence rate as demonstrated in \cite{kovalev2020don}.
Therefore, the L-SVRG provides a more practical and efficient alternative to the SVRG, offering both simple implementation and desirable linear convergence.

\begin{remark}
Notice that it is natural to extend the stochastic gradient $v_k$ to the batch version, i,e., replace $v_k$ defined in \eqref{eq: vr gradient} with
\begin{equation}\label{s5}
	v_{k}=\nabla f_{\mathcal{B}_k}(x_{k})-\nabla f_{\mathcal{B}_k}(w_{k})+\nabla f(w_{k}),
\end{equation}
where $\mathcal{B}_k\subset [n]$ is independently selected at each iteration.
\end{remark}

\subsection{Construction of \texorpdfstring{$B_k$}{}}\label{se:construct Hk}

We consider the stochastic L-BFGS scheme and use a set of correction pairs $\{\left(s_j, y_j\right)\}_{j=1}^m$ with $s_j^\top y_j>0$ to construct $B_k$, where $m>0$ is the memory size. The correction pairs are generated using the Hessian-vector product scheme described in \cite{byrd2016stochastic}.  Specifically, new correction pairs are computed every $r$ iterations, and at most a specified number (denoted by $l$) of the latest computed correction pairs are stored.  We introduce a separate superscript $t$ to denote the number of correction pairs computed so far, where $t$ is the integer division of $k$ by $r$. Denoting the new correction pair as $(s^t, y^t)$, it is generated based on a collection of average iterates, i.e.,
\begin{equation}\label{eq: hess-vec}
s^t=\bar{x}^t-\bar{x}^{t-1}~\text{and}~y^t=\nabla^2 f_{\mathcal{S}^t}\left(\bar{x}^t\right) s^t,
\end{equation}
where $\bar{x}^t:=\frac{1}{r}\sum_{j=k-r+1}^{k} x_j$ and $\mathcal{S}^t$ is randomly sampled from $[n]$.
It is noteworthy that the matrix $\nabla^2 f_{\mathcal{S}^t}\left(\bar{x}^t\right)$ does not need to be constructed explicitly, one can directly compute the Hessian-vector product $\nabla^2 f_{\mathcal{S}^t}\left(\bar{x}^t\right) s^t$, which can save lots of computational and storage costs. Taking
 $\{\left(s_j, y_j\right)\}_{j=1}^m (m\leq l)$ as the latest computed $m$ correction pairs, i.e., $s_j := s^{t-m+j}$ and $y_j := y^{t-m+j}$ for $j\in [m]$, the compact representation of the L-BFGS matrix $B^t$ (see \cite{nocedal1999numerical}) reads as :
\begin{equation} \label{eq: compact_lbfgs}
	B^t:=\sigma_{0}I - \begin{bmatrix}\sigma_{0} S & Y\end{bmatrix}\begin{bmatrix} \sigma_{0}  S^\top S  & L\\ L^\top& -D\end{bmatrix}^{-1}\begin{bmatrix}\sigma_0 S^\top \\ Y^\top \end{bmatrix},~\sigma_0:=\frac{(y_m)^\top y_m}{(y_m)^\top s_m},
\end{equation}
where
\begin{equation*}
	S:=[s_{1}, \dots, s_{m}]\in \rbb^{d\times m},\;  Y:=[y_{1}, \dots, y_{m}]\in \rbb^{d\times m},\;D:=\text{diag}[s^\top_{1} y_{1} \dots s^\top_{m} y_{m}] \in \rbb^{m\times m},
\end{equation*}
and $L\in \mathbb{R}^{m \times m}$ is defined as
\begin{equation*}
	(L)_{i,j}=\left\{
	\begin{aligned}
		&s_{i}^\top y_{j} & \text{~if~} i> j,\\
		&0 & \text{~otherwise.}
	\end{aligned}
	\right.
\end{equation*}
At the initial stage (i.e., $k < r$), we set $B_k = I$. For $ k \geq r$, we set $B_k = B^t$ and update it every $r$ iterations.  In addition, the matrix $B^t$ does not need to be explicitly computed, as we will discuss in detail in Section \ref{sec: ssn}.

It is worth mentioning an alternative stochastic L-BFGS scheme in \cite{berahas2016multi}, which generates a correction pair $(s_k, y_k)$ at every iteration. To be concrete, a small batch $S_k \subset [n]$ is generated at the $k$-th iteration such that $O_k:=S_k \cap S_{k-1} \neq \emptyset$.
Then the correction pair $(s_k, y_k)$ is computed via
\begin{equation*}
	s_{k}=x_{k}-x_{k-1}, \quad
	y_{k}=\nabla f_{O_k}(x_{k})-\nabla f_{O_k}(x_{k-1}).
\end{equation*}
Note that the construction of $y_{k}$ in the above L-BFGS scheme does not require extra gradient computation, as $O_k$ is a subset of both $S_k$ and $S_{k-1}$.
However, this approach may encounter practical challenges, such as the inability to ensure the independence of the samples $\left\{S_k\right\}$ or the insufficiency of the overlap set $O_k$ to offer valuable Hessian information. In contrast, the stochastic L-BFGS Hessian-vector product scheme \eqref{eq: hess-vec}, as highlighted in \cite{byrd2016stochastic}, addresses this limitation and achieves a stable Hessian approximation by decoupling the computation of $s$ and the construction of $y$.
It also allows for the incorporation of new Hessian information periodically, and this flexibility enhances the performance and efficiency of the Hessian-vector product scheme \eqref{eq: hess-vec}.

\begin{remark}
	It is important to note that even the deterministic version of L-BFGS does not achieve locally superlinear convergence; generally, only linear convergence rate can be attained. This limitation can be attributed to the failure to meet the Dennis-Moré criterion \cite{dennis1974characterization}, which is both necessary and sufficient for the superlinear convergence of quasi-Newton methods \cite{nocedal1999numerical}. Therefore, in Section \ref{sec: analysis}, we focus on studying the global convergence behavior of the proposed method.
\end{remark}

\subsection{A single-loop stochastic proximal quasi-Newton method for (\ref{eq: main prob})}\label{se:algorithm-table}
Based on the discussions in Sections \ref{se:construct vk} and \ref{se:construct Hk}, we propose a single-loop stochastic proximal quasi-Newton method for solving \eqref{eq: main prob}, as summarized in Algorithm \ref{algo: spbfgs}.

\begin{algorithm}[htpb]
	\caption{A single-loop stochastic proximal quasi-Newton method for \eqref{eq: main prob}}
	\label{algo: spbfgs} 	
	\begin{algorithmic}
		\Require initial points $x_0 = w_0\in \rbb^d$; step sizes $\{\eta_k\}_{k\ge 0}$; probability parameter $p\in (0, 1)$; Hessian update frequency $r>0$.
		\State Set $t=0, \bar{x}_{0} = 0$.
		\For{$k = 0, 1, 2, \cdots$}
		\If{$\mathrm{mod}(k,r)=0$ and $k\ge r$}
		\State Set $t = t + 1$ and $\bar{x}^{t}=\frac{1}{r}\sum_{j=k-r+1}^{k}x_{j}$.
    \State Compute $s^{t}=\bar{x}^{t}-\bar{x}^{t-1}$.
		\State Sample $\mathcal{S}^{t}\subset[n]$ and compute $y^{t}=\nabla^2 f_{\mathcal{S}^t}(\bar{x}^{t})s^{t}$.
		\State Update  the correction pairs $\{\left(s_j, y_j\right)\}_{j=1}^m$ with $(s^t, y^t)$.
		\EndIf
		\State Sample an index $i_k$ uniformly at random from $[n]$.
		\State Compute $v_{k}$ according to \eqref{eq: vr gradient} or \eqref{s5}.
		\If{$k < r$}
		\State $x_{k+1} = \p_{\eta_k,\;h}(x_{k}-\eta_k v_{k})$.
		\Else
		\State Update the iterate $x_{k+1}$ by solving \eqref{eq:subproblem} with $B_k:=B^t$ as defined by \eqref{eq: compact_lbfgs}.
		\EndIf
		\State Update the reference point $w_{k+1}=\begin{cases}x_k&\text{with probability}\; p,\\w_k&\text{with probability}\; 1-p.\end{cases}$
		\EndFor
	\end{algorithmic}
\end{algorithm}

It is easy to see that if the Hessian approximation $B_k$ is fixed as the identity matrix $I_d$, Algorithm \ref{algo: spbfgs} simplifies into a proximal variant of the L-SVRG, which can be considered as an extension of the L-SVRG for solving the nonsmooth problem \eqref{eq: main prob}. We present the proximal L-SVRG method in Algorithm \ref{algo: plsvrg}. Additionally, note that Algorithm \ref{algo: spbfgs} can be generalized by generating the stochastic gradient $v_k$ via other VR methods, such as the SAGA \cite{defazio2014saga} and the SEGA \cite{hanzely2018sega}. As to be shown in Section \ref{se:generalization}, the resulting generalized algorithms exhibit linear convergence rates.

Like other proximal Newton-type methods, Algorithm \ref{algo: spbfgs} requires to solve the subproblem \eqref{eq:subproblem} to proceed, which generally lacks a closed-form solution.
To tackle this challenge, we introduce a rapid inner solver in Section \ref{sec: ssn} that efficiently solves \eqref{eq:subproblem}, thereby significantly reducing the overall computational cost. It's worth noting that this solver can be leveraged for any proximal Newton-type methods that require solving subproblems akin to \eqref{eq:subproblem}, consequently improving their numerical efficiency as well.

\begin{algorithm}[htpb]
	\caption{Proximal Loopless SVRG for solving \eqref{eq: main prob}}
	\label{algo: plsvrg}
	\begin{algorithmic}
	\Require initial points $x_0 = w_0 \in \rbb^d$; step sizes $\{\eta_k\}_{k\ge 0}$; probability parameter $p\in (0, 1)$.
		\For{$k = 0, 1, \dots$}
		\State Sample an index $i_k$ uniformly at random from $[n]$.
		\State Compute $v_{k}$ according to \eqref{eq: vr gradient} or \eqref{s5}.
		\State Update $x_{k+1} = \p_{\eta_k,\;h}(x_{k}-\eta_k v_{k})$ and $w_{k+1}=\begin{cases}x_k&\text{with probability}\; p,\\w_k&\text{with probability}\; 1-p.\end{cases}$
		\EndFor
	\end{algorithmic}
\end{algorithm}

\section{Convergence analysis} \label{sec: analysis}
We present our convergence analyses in this section.
Specifically, in Section \ref{se:con_alg1} we focus on Algorithm \ref{algo: spbfgs} and demonstrate its globally linear convergence. The convergence of Algorithm \ref{algo: plsvrg} is established in Section \ref{se:con_plsvrg}.
Finally, in Section \ref{se:generalization} we generalize Algorithm \ref{algo: spbfgs} and provide corresponding convergence results. Note that while the results established in this section focus on the case \eqref{eq: vr gradient}, they can be readily extended to the batch version \eqref{s5}.

\subsection{Convergence of Algorithm \ref{algo: spbfgs}}\label{se:con_alg1}
In this subsection, we shall show that the iterative sequence generated by Algorithm \ref{algo: spbfgs} converges linearly towards the optimal solution in expectation.
We start by making the following assumption, which is widely adopted in the related literature.
\begin{assumption}
	\label{ass: fh}
     Each component $f_i$ is twice continuously differentiable, and there exist constants $\mu, L>0$ such that
\begin{equation*}
	\mu I \preceq \nabla^2 f_{\mathcal{S}}(x) \preceq L I
\end{equation*}
for any $x \in \rbb^d$ and $\mathcal{S} \subset [n]$.
\end{assumption}
Note that Assumption \ref{ass: fh} implies that each $f_i$ and $f$ are $\mu$-strongly convex and $L$-smooth, and the objective $F$ is $\mu$-strongly convex, thereby ensuring the existence of a unique minimizer $x^*$ of \eqref{eq: main prob}.

Under Assumption \ref{ass: fh}, it can be established that the eigenvalues of the Hessian approximations defined by \eqref{eq: compact_lbfgs} are uniformly bounded both from above and away from zero as summarized in the following lemma.
\begin{lemma}[{\cite[Lemma 3.1]{byrd2016stochastic}}] \label{lemma: bounded hessian}
	Let Assumption \ref{ass: fh} hold, then there exist constants $0 < m_1 \leq m_2$ such that the Hessian approximations $\{B^t\}_{t\ge 1}$ defined by \eqref{eq: compact_lbfgs} satisfies
	\begin{equation*}
		m_1 I \preceq B^{t} \preceq m_2 I, ~\forall t\geq 1.
	\end{equation*}
\end{lemma}

The upcoming lemma demonstrates that, in the L-SVRG scheme \eqref{eq: vr gradient}, if both $\{x_k\}$ and $\{w_k\}$ converge to $x^*$, then the variances of $\{v_k\}$ converges to zero.
This result plays a crucial role in establishing linear convergence for Algorithm \ref{algo: spbfgs}  with a constant step size. 
\begin{lemma} \label{lem: bv}
	Consider problem \eqref{eq: main prob} and suppose that $f_i$ is convex and $L$-smooth for each $i\in [n]$. Let $v_k$ be defined as in \eqref{eq: vr gradient}, then for any $k\ge 0$, we have $\e_k[v_k] = \nabla f(x_k)$ and
	\begin{equation} \label{eq: lsvrg-1}
		\mathbb{E}_k\left[\left\|v_{k}-\nabla f(x_{k})\right\|^{2}\right]\leq4L[F(x_{k})-F^{*}+F(w_{k})-F^{*}].
	\end{equation}
\end{lemma}

\begin{proof}
	Firstly, it is easy to verify $\e_k[v_k] = \nabla f(x_k)$ by noting that $\e_k[\nabla f_{i_k}(x_k)]=\nabla f(x_k)$ and $\e_k[\nabla f_{i_k}(w_k)]=\nabla f(w_k)$. Moreover, from the fact that $\e_k\left[\|\zeta-\e_k [\zeta]\|^2\right]=\e_k\left[\|\zeta\|^2\right]-\|\e_k [\zeta]\|^2$ for any random vector $\zeta$, we have
	\begin{equation}\label{eq: lsvrg-2}
		\begin{aligned}
			\e_k\left[\left\|v_k-\nabla f\left(x_{k}\right)\right\|^2\right] & =\e_k\left[\left\|\nabla f_{i_k}\left(x_{k}\right)-\nabla f_{i_k}(w_k)+\nabla f(w_k)-\nabla f\left(x_k\right)\right\|^2\right] \\
			& =\e_k \left[\left\|\nabla f_{i_k}\left(x_k\right)-\nabla f_{i_k}(w_k)\right\|^2\right]-\left\|\nabla f\left(x_k\right)-\nabla f(w_k)\right\|^2 \\
			& \leq \e_k \left[\left\|\nabla f_{i_k}\left(x_k\right)-\nabla f_{i_k}\left(x^*\right)+\nabla f_{i_k}\left(x^*\right)-\nabla f_{i_k}(w_k)\right\|^2\right].
		\end{aligned}
	\end{equation}
	Apply Young's inequality on the RHS of \eqref{eq: lsvrg-2}, and then apply Lemma \ref{lem: pre_bv}, we obtain 
	\begin{equation*}
		\begin{aligned}
			& \e_k\left[\left\|v_k-\nabla f\left(x_{k}\right)\right\|^2\right]\\
      \leq & 2\e_k\left[ \left\|\nabla f_{i_k}\left(x_k\right)-\nabla f_{i_k}\left(x^*\right)\right\|^2 \right]+ 2\e_k \left[\left\|\nabla f_{i_k}(w_k)-\nabla f_{i_k}\left(x^*\right)\right\|^2\right] \\
			=&\frac{2}{n} \sum_{i=1}^n \left\|\nabla f_i\left(x_k\right)-\nabla f_i\left(x^*\right)\right\|^2+\frac{2}{n} \sum_{i=1}^n \left\|\nabla f_i(w_k)-\nabla f_i\left(x^*\right)\right\|^2 \\
			\leq & 4 L\left[F\left(x_k\right)-F^*+F(w_k)-F^*\right],
		\end{aligned}
	\end{equation*}
	which completes the proof.
\end{proof}

We are now ready to unveil the convergence behavior of the iterates $\left\{x_k\right\}$ produced by Algorithm \ref{algo: spbfgs} towards the optimal solution $x^*$. For this purpose, we introduce a Lyapunov sequence $\{\phi_k\}$ defined as
\begin{equation} \label{eq: lyapunov}
	\phi_k:=\left\|x_k-x^*\right\|_{B_k}^2+A\eta_k^2\left[F\left(w_k\right)-F^*\right]+2 \eta_k\left[F\left(x_k\right)-F^*\right], \forall k\geq 0,
\end{equation}
where the constant $A>0$ is introduced to balance the term $F(w_k)-F^*$ via the update rule of reference points $\{w_k\}$. In the following theorem, we show that, with an appropriate constant step size $\eta_k\equiv\eta>0$ and Hessian update frequency $r>0$, the sequence $\left\{\e[\phi_k]\right\}$ converges to $0$ linearly, which implies that the distance between $\{x_k\}$ and $x^*$ diminishes linearly to $0$ in expectation.

\begin{theorem} \label{thm: main}
	Let Assumption \ref{ass: fh} hold.
	Choose a positive constant $A$ such that $A>\frac{8L\sqrt{m_2}}{p m_1\sqrt{m_1}}$, where $m_1, m_2>0$ are constants satisfying $m_1 I \preceq B^t \preceq m_2 I$ for any $t\geq 1$ as indicated by Lemma \ref{lemma: bounded hessian}. Let $\{x_k\}$ be the iterates generated by Algorithm \ref{algo: spbfgs} with constant step size $\eta>0$ and Hessian update frequency $r>0$ satisfying
	\begin{equation} \label{eq: eta_choice}
		\eta \leq \min \left\{\frac{m_1}{L}, \frac{2 m_1 m_2 \sqrt{m_1}}{8 L m_2\sqrt{m_2}+2 \mu m_1\sqrt{m_1}+A p m_1 m_2 \sqrt{m_1}}\right\}
	\end{equation}
and
     \begin{equation*}
     	\left(\frac{m_2}{m_1}\right)^{1/r}\rho<1,
     \end{equation*}
 where $\rho:=\max \left\{1-\frac{\eta \mu}{m_2}, \frac{8L\sqrt{m_2}}{A m_1\sqrt{m_1}}+1-p\right\}.$
	 Then we have $\rho\in (0, 1)$, and for any $k\ge r$, it holds that
	\begin{equation*}
		\mathbb{E}\left[\phi_k\right] \leq \left(\frac{m_2}{m_1}\rho^r\right)^{\lfloor \frac{k}{r} \rfloor} \phi_{0}.
	\end{equation*}
\end{theorem}

\begin{proof}
	Apply Lemma \ref{lem: key lem} with $x=x_{k}$, $v=v_{k}$, $B=B_k$, and $y=x^{*}$, then we have $x^{+}=x_{k+1}$ and hence $g=g_{k}:=\frac{1}{\eta}(x_{k}-x_{k+1})$ and $\Delta = \Delta_k:= v_k - \nabla f(x_k)$, which gives
	\begin{equation}\label{eq: key lem}
		F^*\geq F(x_{k+1})+g_{k}^{\top}B_{k}(x^{*}-x_{k})+\Delta_{k}^{\top}(x_{k+1}-x^{*})+\eta\|g_k\|_{B_k}^2-\frac{L\eta^2}{2}\|g_k\|^2+\frac{\mu}{2}\|x_{k}-x^{*}\|^{2}.
	\end{equation}
	Using \eqref{eq: key lem}, we expand $\left\|x_{k+1}-x^*\right\|_{B_k}^2=\|x_{k+1}-x_k+x_k-x^*\|_{B_k}^2$ as follows
	\begin{equation} \label{eq: x_expand-1}
		\begin{aligned}
			\left\|x_{k+1}-x^*\right\|_{B_k}^2 = & \left\|x_k-x^*\right\|_{B_k}^2+2\left(x_k-x^*\right)^{\top} B_k\left(x_{k+1}-x_k\right)+\left\|x_{k+1}-x_k\right\|_{B_k}^2 \\
			= & \left\|x_k-x^*\right\|_{B_k}^2-2 \eta g_k^{\top} B_k\left(x_k-x^*\right)+\eta^2\left\|g_k\right\|_{B_k}^2 \\
			\leq & \left\|x_k-x^*\right\|_{B_k}^2+2 \eta\left[F^*-F\left(x_{k+1}\right)\right]-2 \eta \Delta_k^{\top}\left(x_{k+1}-x^*\right) - \eta^2 \|g_k\|_{B_k}^2 \\
      &+ L\eta^3 \|g_k\|^2 -\mu \eta\left\|x_k-x^*\right\|^2.
		\end{aligned}
	\end{equation}
	Since $ m_1 I \preceq B_k \preceq m_2 I$ and $\eta\leq \frac{m_1}{L}$, we have from \eqref{eq: x_expand-1} that
	\begin{equation}\label{eq: x_expand}
		\begin{aligned}
			\left\|x_{k+1}-x^*\right\|_{B_k}^2 \leq & \left\|x_k-x^*\right\|_{B_k}^2+2 \eta\left[F^*-F\left(x_{k+1}\right)\right]-2 \eta \Delta_{k}^{\top}\left(x_{k+1}-x^*\right)\\
      &-\eta^2(m_1-L\eta)\|g_k\|^2 - \frac{\mu \eta}{m_2}\left\|x_k-x^*\right\|_{B_k}^2\\
			\leq & \left(1-\frac{\mu \eta}{m_2}\right)\left\|x_k-x^*\right\|_{B_k}^2+2 \eta\left[F^*-F\left(x_{k+1}\right)\right]-2 \eta \Delta_{k}^{\top}\left(x_{k+1}-x^*\right).
		\end{aligned}
	\end{equation}
	
	Next, we bound the last term  $-2 \eta \Delta_{k}^{\top}\left(x_{k+1}-x^*\right)$ on the RHS of (\ref{eq: x_expand}). To this end, we define
	\begin{equation*}
		\bar{x}_{k+1}=\p_{\eta,\;h}^{B_k}\left(x_k-\eta H_k \nabla f\left(x_k\right)\right),
	\end{equation*}
	and split $-2 \eta \Delta_{k}^{\top}\left(x_{k+1}-x^*\right)$ into two parts:
	\begin{equation} \label{eq: ip_delta}
		-2 \eta \Delta_k^{\top}\left(x_{k+1}-x^*\right)= - 2 \eta \Delta_k^{\top}\left(x_{k+1}-\bar{x}_{k+1}\right)-2 \eta \Delta_k^{\top}\left(\bar{x}_{k+1}-x^*\right).
	\end{equation}
	Taking conditional expectation $\e_k[\cdot]$ on $\Delta_k^{\top}\left(\bar{x}_{k+1}-x^*\right)$ gives
	\begin{equation*}
		\mathbb{E}_k\left[\Delta_k^{\top}\left(\bar{x}_{k+1}-x^*\right)\right]=\left(\mathbb{E}_k \left[\Delta_k\right]\right)^{\top}\e_k\left[\bar{x}_{k+1}-x^*\right]=0,
	\end{equation*}
    where the first equality holds due to the independence between $i_k$ and $\mathcal{S}_t$ \cite[Section V.2]{feller1971introduction}, and the second equality follows since $\e_k[\Delta_k]=0$ from Lemma \ref{lem: bv}.
	For the term $-\Delta_k^{\top}\left(x_{k+1}-\bar{x}_{k+1}\right)$ in \eqref{eq: ip_delta}, applying (\ref{eq:nonexpansion}) and considering the fact that $\|H_k\|\leq \frac{1}{m_1}$, we bound it as follows
	\begin{equation*}
		\begin{aligned}
			-\Delta_k^{\top}\left(x_{k+1}-\bar{x}_{k+1}\right) \leq & \left\|\Delta_k\right\| \cdot\left\|\p_{\eta,\;h}^{B_k}\left(x_k-\eta H_k v_k\right)-\p_{\eta,\;h}^{B_k}\left(x_k-\eta H_k \nabla f\left(x_k\right)\right)\right\| \\
			\leq & \sqrt{\frac{m_2}{m_1}}\left\|\Delta_k\right\|\left\|\left(x_k-\eta H_k v_k\right)-\left(x_k-\eta H_k \nabla f\left(x_k\right)\right)\right\| \\
			= & \sqrt{\frac{m_2}{m_1}} \eta\left\|\Delta_k\right\| \cdot\left\|H_k \Delta_k\right\| \\
			\leq & \frac{\sqrt{m_2}}{m_1\sqrt{m_1}} \eta\left\|\Delta_k\right\|^2.
		\end{aligned}
	\end{equation*}
	Therefore, taking conditional expectation $\e_k[\cdot]$ on both sides of \eqref{eq: ip_delta} and applying Lemma \ref{lem: bv}, we have
	\begin{equation} \label{eq: e_ip_delta}
		\begin{aligned}
		-2 \eta \mathbb{E}_k[\Delta_k^{\top}\left(x_{k+1}-x^*\right)]
		&\leq \frac{2\sqrt{m_2}}{m_1\sqrt{m_1}} \eta^2 \mathbb{E}_k[\left\|\Delta_k\right\|^2]\\
		& \leq \frac{8L\sqrt{m_2}}{m_1\sqrt{m_1}} \eta^2\left[F\left(x_k\right)-F^*+F\left(w_k\right)-F^*\right].
		\end{aligned}
	\end{equation}
Then, taking conditional expectation $\e_k[\cdot]$ on both sides of \eqref{eq: x_expand} and it follows from \eqref{eq: e_ip_delta} that
	\begin{equation} \label{eq: e_x_expand}
		\begin{aligned}
			\mathbb{E}_k\left[\left\|x_{k+1}-x^*\right\|_{B_k}^2\right] \leq & \left(1-\frac{\mu \eta}{m_2}\right)\left\|x_k-x^*\right\|_{B_k}^2-2 \eta \mathbb{E}_k\left[F\left(x_{k+1}\right)-F^*\right]\\
			&+\frac{8L\sqrt{m_2}}{m_1\sqrt{m_1}} \eta^2\left[F\left(x_k\right)-F^*+F\left(w_k\right)-F^*\right].
		\end{aligned}
	\end{equation}
	
	On the other hand, from the update rule \eqref{eq:update_w}, we have
	\begin{equation} \label{eq: lsvrg_w}
		\mathbb{E}_k\left[F\left(w_{k+1}\right)-F^*\right]=(1-p)\left[F\left(w_k\right)-F^*\right]+p\left[F\left(x_k\right)-F^*\right].
	\end{equation}
	Adding $A \eta^2\mathbb{E}_k\left[F\left(w_{k+1}\right)-F^*\right]$ on both sides of \eqref{eq: e_x_expand}, then it follows from (\ref{eq: lsvrg_w}) that
	\begin{equation}\label{eq: add_A}
		\begin{aligned}
			& \mathbb{E}_k\left[\left\|x_{k+1}-x^*\right\|_{B_k}^2+A\eta^2\left(F\left(w_{k+1}\right)-F^*\right)\right] \\
			\leq & \left(1-\frac{\mu \eta}{m_2}\right)\left\|x_k-x^*\right\|_{B_k}^2-2 \eta \mathbb{E}_k\left[F\left(x_{k+1}\right)-F^*\right]+\left(\frac{8L\sqrt{m_2}}{m_1\sqrt{m_1}} +A p\right)\eta^2\left[F\left(x_k\right)-F^*\right] \\
			& \quad+\left(\frac{8L\sqrt{m_2}}{A m_1\sqrt{m_1}}+1-p\right)A\eta^2\left[F\left(w_k\right)-F^*\right].
		\end{aligned}
	\end{equation}
 	Note that $\eta > 0$ guarantees $1-\frac{\mu \eta}{m_2} < 1$, and the condition $A>\frac{8L\sqrt{m_2}}{p m_1\sqrt{m_1}}$ ensures that $0<\frac{8L\sqrt{m_2}}{A m_1\sqrt{m_1}}+1-p < 1$. Therefore, we have $\rho \in (0, 1)$. With the definition of $\rho$, we obtain from \eqref{eq: add_A} that
	\begin{equation}\label{eq: add_A_1}
		\begin{aligned}
			& \mathbb{E}_k\left[\left\|x_{k+1}-x^*\right\|_{B_k}^2+A\eta^2\left(F\left(w_{k+1}\right)-F^*\right)+2 \eta\left(F\left(x_{k+1}\right)-F^*\right)\right] \\
			\leq& \rho\left[\left\|x_k-x^*\right\|_{B_k}^2+A\eta^2\left(F\left(w_k\right)-F^*\right)+\frac{\eta^2}{\rho}\left(\frac{8L\sqrt{m_2}}{m_1\sqrt{m_1}}+A p\right)\left(F\left(x_k\right)-F^*\right)\right].
		\end{aligned}
	\end{equation}
	Since	
	\begin{equation*}
		\eta \leq \frac{2 m_1 m_2 \sqrt{m_1}}{8 L m_2\sqrt{m_2}+2 \mu m_1\sqrt{m_1}+A p m_1 m_2 \sqrt{m_1}},
	\end{equation*}
	we have that
		\begin{equation}\label{eq: eta-ineq}
		\eta \leq \frac{2-\frac{2 \eta \mu}{m_2}}{\frac{8L\sqrt{m_2}}{m_1\sqrt{m_1}}+A p}.
	\end{equation}
Then, it follows from $\rho\geq 1-\frac{\mu \eta}{m_2}$ and (\ref{eq: eta-ineq}) that
 \begin{equation}\label{eq: eta_cond}
     \frac{\eta^2}{\rho}\left(\frac{8L\sqrt{m_2}}{m_1\sqrt{m_1}}+A p\right) \leq 2 \eta.
 \end{equation}
	Hence, combining \eqref{eq: add_A_1} with \eqref{eq: eta_cond} we have
	\begin{equation}\label{eq: phi_k}
		\begin{aligned}
			& \mathbb{E}_k\left[\left\|x_{k+1}-x^*\right\|_{B_k}^2+A\eta^2\left(F\left(w_{k+1}\right)-F^*\right)+2 \eta\left(F\left(x_{k+1}\right)-F^*\right)\right] \\
			\leq&\rho\left[\left\|x_k-x^*\right\|_{B_k}^2+A\eta^2\left(F\left(w_k\right)-F^*\right)+2\eta\left(F\left(x_k\right)-F^*\right)\right].
		\end{aligned}
	\end{equation}

    For any $k\ge r$, from the condition $m_1 I \preceq B_k \preceq m_2 I$ we have
	\begin{equation}\label{eq: thm1-1}
		\left\|x_{k+1}-x^*\right\|_{B_k}^2 \ge m_1\left\|x_{k+1}-x^*\right\|^2 \ge \frac{m_1}{m_2}\left\|x_{k+1}-x^*\right\|_{B_{k+1}}^2.
	\end{equation}
	Also notice that within any consecutive $r$ iterations, $H_k$ only changes once.
	Therefore, it follows from \eqref{eq: phi_k} and \eqref{eq: thm1-1} that
	\begin{equation}\label{eq: thm1-2}
		\e\left[\phi_k\right] \leq \left(\frac{m_2}{m_1}\rho^r\right) \e\left[\phi_{k-r}\right]\leq \left(\frac{m_2}{m_1}\rho^r\right)^{\lfloor \frac{k}{r} \rfloor} \e\left[\phi_{k-r\lfloor \frac{k}{r} \rfloor}\right],
	\end{equation}
 where $\phi_k$ is the Lyapunov function defined by \eqref{eq: lyapunov}.
 On the other hand, since $B_k=I$ for all $k < r$, we have
	\begin{equation*}
		\phi_k=\left\|x_k-x^*\right\|^2+A\eta^2\left[F\left(w_k\right)-F^*\right]+2 \eta\left[F\left(x_k\right)-F^*\right],;\forall k<r,
	\end{equation*}
	which together with \eqref{eq: phi_k} gives
	\begin{equation} \label{eq: thm1-0}
		\e\left[\phi_{k}\right] \leq \rho\cdot\e\left[\phi_{k-1}\right]\leq \rho^k \phi_0\leq \phi_0, ~\forall k<r.
	\end{equation}
Note that $k-r\lfloor \frac{k}{r} \rfloor < r$. Then combining \eqref{eq: thm1-2} with \eqref{eq: thm1-0} leads to
	\begin{equation}\label{eq: phi_convergence}
		\e\left[\phi_k\right] \leq \left(\frac{m_2}{m_1}\rho^r\right)^{\lfloor \frac{k}{r} \rfloor} \phi_{0},
	\end{equation}
	which completes the proof.
\end{proof}

Theorem \ref{thm: main} shows that the the sequence of Lyapunov functions $\{\phi_k\}$ converges to $0$ linearly in expectation since $\left(m_2/m_1\right)^{1/r}\rho<1$. As a direct consequence, we have that $\{\|x_k-x^*\|_{B_k}\}$ converges to $0$ linearly, which together with $m_1 I\preceq B_k$ implies that $\{x_k\}$ converges to $x^*$ linearly. Moreover, the two sequences of the objective values $\{F(x_k)\}$ and $\{F(w_k)\}$ also converge to $F^*$ at linear rates.

\subsection{Convergence of Algorithm \ref{algo: plsvrg}}\label{se:con_plsvrg}
In this subsection, we establish a linear convergence result and complexity analysis for Algorithm \ref{algo: plsvrg}. To this end, we first note that Algorithm \ref{algo: plsvrg} can be regarded as a special case of Algorithm \ref{algo: spbfgs} wherein the matrix $B_k$ remains fixed as the identity matrix.
Therefore, by setting $m_1 = m_2 = 1$ in Theorem \ref{thm: main}, we can derive the following convergence result for Algorithm \ref{algo: plsvrg}.

\begin{theorem} \label{thm: plsvrg}
	Let Assumption \ref{ass: fh} hold.
    Let $\{x_k\}$ be the iterates generated by Algorithm \ref{algo: plsvrg} with a constant step size $\eta>0$ satisfying
	\begin{equation} \label{eq: eta_requirement}
		0 < \eta \leq \min \left\{\frac{1}{L}, \frac{1}{9 L+ \mu}\right\},
	\end{equation}
	we then have
	\begin{equation}\label{eq: plsvrg_convergence}
		\e\left[\phi_k\right] \leq \max \left\{1-\eta \mu, 1-\frac{p}{5}\right\}^k \phi_0,
	\end{equation}
	where $\phi_k$ is the Lyapunov function defined in \eqref{eq: lyapunov} with $A = \frac{10 L}{p}$ and $B_k = I$, i.e.,
	\begin{equation*}
		\phi_k := \left\|x_k-x^*\right\|^2+\frac{10 L}{p}\eta^2\left[F\left(w_k\right)-F^*\right]+2 \eta\left[F\left(x_k\right)-F^*\right].
	\end{equation*}
\end{theorem}

\begin{proof}
	Setting $m_1 = m_2 = 1$ in Theorem \ref{thm: main}, we have that the step size requirement \eqref{eq: eta_choice} reduces to
	\begin{equation*}
		0<\eta \leq \min \left\{\frac{1}{L}, \frac{2}{8 L+2 \mu+A p}\right\}~\text{and}~\rho = \max \left\{1-\eta \mu, \frac{8L}{A}+1-p\right\},
	\end{equation*}
which, together with $A = \frac{10L}{p}$, gives \eqref{eq: eta_requirement} and ensures that $\rho \in (0, 1)$, so that the condition $\left(\frac{m_2}{m_1}\right)^{1/r}\rho<1$ is satisfied. %
Then, substituting $B_k=I$ and $A= \frac{10L}{p}$ in \eqref{eq: phi_k}, and noting that $r=1$ in Algorithm \ref{algo: plsvrg}, we obtain
	\begin{equation*}
		\e\left[\phi_k\right] \leq \rho^k \phi_0 = \max \left\{1-\eta\mu,1-\frac{p}{5}\right\}^k \phi_0,
	\end{equation*}
	which completes the proof.
\end{proof}

Theorem \ref{thm: plsvrg} implies the following iteration complexity result for Algorithm \ref{algo: plsvrg}.

\begin{corollary}
	\label{coro: 1}
	Consider the same setting as that in Theorem \ref{thm: plsvrg}.
	Suppose we choose $\eta = \frac{1}{10L}$.
	Then for any $\epsilon>0$, we have $\e\left[\phi_k\right] \leq \epsilon\phi_0$ if
	\begin{equation} \label{eq: k_complexity}
		k \geq 10\left(\frac{1}{p}+\frac{L}{\mu}\right) \log\left(\frac{1}{\epsilon}\right).
	\end{equation}
	Furthermore, if we select $p\in\left[\min \left\{\frac{c}{n}, \frac{c \mu}{L}\right\}, \max \left\{\frac{c}{n}, \frac{c \mu}{L}\right\}\right]$ for any constant $c>0$, then the expected total number of stochastic gradient oracles is
	\begin{equation} \label{eq: oracle_complexity}
		O\left(\left(n+\frac{L}{\mu}\right) \log \frac{1}{\epsilon}\right).
	\end{equation}
\end{corollary}

\begin{proof}
	To start off, note that $\eta = \frac{1}{10L}$ satisfies the step size requirement \eqref{eq: eta_requirement} since $\mu \leq L$, hence plugging $\eta = \frac{1}{10L}$ into \eqref{eq: plsvrg_convergence} gives
	\begin{equation} \label{eq: thm2-1}
		\e\left[\phi_k\right] \leq \rho^k \phi_0,
	\end{equation}
	with $\rho = \max \left\{1-\frac{\mu}{10 L}, 1-\frac{p}{5}\right\}$. From \eqref{eq: thm2-1} and the inequality $\log(\beta)\leq \beta-1$ for any $\beta>0$, we know that $\e\left[\phi_k\right] \leq \epsilon \phi_0$ as long as
	\begin{equation} \label{eq: k_req}
		k \geq \frac{1}{1-\rho} \log \left(\frac{1}{\epsilon}\right).
	\end{equation}
	In addition, notice that
	\begin{equation*}
		10 \left(\frac{L}{\mu} + \frac{1}{p}\right) > \max\left\{\frac{10L}{\mu}, \frac{5}{p}\right\} = \frac{1}{1-\rho}.
	\end{equation*}
	Therefore, choosing $k = 10 \left(\frac{1}{p}+\frac{L}{\mu}\right)\log \left(\frac{1}{\epsilon}\right)$ satisfies \eqref{eq: k_req}, which proves \eqref{eq: k_complexity}.
	
	To show \eqref{eq: oracle_complexity}, note that at each iteration, Algorithm \ref{algo: plsvrg} calls $O(1 + pn)$ stochastic gradient in expectation.
	Combining it with the iteration complexity \eqref{eq: k_complexity} leads to the expected total number of stochastic gradient oracles
	\begin{equation*}
		O\left(\left(\frac{1}{p}+n+\frac{L}{\mu}+\frac{L p n}{\mu}\right) \log \frac{1}{\epsilon}\right).
	\end{equation*}
	Therefore, given any $c>0$, any choice of $p \in [\min \{c / n, c \mu / L\}, \max \{c / n, c \mu / L\}]$ leads to the total complexity \eqref{eq: oracle_complexity}, which completes the proof.
\end{proof}

\begin{remark}
	The complexity result $O\left(\left(n+\frac{L}{\mu}\right) \log \frac{1}{\epsilon}\right)$ of Algorithm \ref{algo: plsvrg} in terms of stochastic gradient oracles aligns with the one of the L-SVRG presented in \cite{kovalev2020don}. Consequently, we generalize the L-SVRG to a proximal variant and establish its linear convergence by leveraging the techniques developed in the proof of Theorem \ref{thm: main}.
\end{remark}

\subsection{Generalization of Algorithm \ref{algo: spbfgs}}\label{se:generalization}
As mentioned in Section \ref{se:algorithm-table}, it is possible to generalize Algorithm \ref{algo: spbfgs} by defining the stochastic gradient $v_k$ using VR methods different from the L-SVRG. In this subsection, we establish linear convergence results for these generalized versions.
To this end, recall that in the proof of Theorem \ref{thm: main}, we resort to two properties of $v_k$ defined by the L-SVRG scheme \eqref{eq: vr gradient}, namely, \eqref{eq: lsvrg-1} and \eqref{eq: lsvrg_w}.
Note that these two inequalities are special cases of the following general form:
\begin{equation} \label{eq: ass_vr}
	\left\{
	\begin{aligned}
		& \e_k\left[\left\|v_k-\nabla f\left(x_k\right)\right\|^2\right] \leq a\left(F\left(x_k\right)-F^*\right)+b \xi_k, \\
		& \e_k\left[\xi_{k+1}\right] \leq(1-\gamma) \xi_k+c\left(F\left(x_k\right)-F^*\right),
	\end{aligned}
\right.
\end{equation}
where $a, b, c\ge0$ and $\gamma\in(0, 1]$ are constants unrelated to the index $k$, and $\{\xi_k\}_{k\ge 0}$ is a sequence of random variables.
Indeed, it is easy to verify that $v_k$ defined by the L-SVRG \eqref{eq: vr gradient} satisfies \eqref{eq: ass_vr} with $a = b = 4L$, $c=\gamma = p$, and $\xi_k = F(w_{k}) - F^*$.

In the remaining part of this subsection, rather than specifying an explicit expression on $v_k$, we make the following assumption concerning its general properties.

\begin{assumption} \label{ass: vr}
	For any $k\ge 0$, the stochastic gradient $v_k$ is an unbiased estimator of $\nabla f(x_k)$, i.e., $\e_k[v_k] = \nabla f(x_k)$.
	Furthermore, there exist constants $a, b, c\ge 0$ and $\gamma\in(0, 1]$, and a sequence of random variables $\{\xi_k\}_{k\ge 0}$ such that \eqref{eq: ass_vr} is satisfied.
\end{assumption}

It can be demonstrated that, in addition to the L-SVRG \cite{kovalev2020don}, the stochastic gradient $v_k$ generated by other VR methods like the SAGA  \cite{defazio2014saga} and the SEGA \cite{hanzely2018sega} also satisfies Assumption \ref{ass: vr}.
For instance, in the SAGA, each component $f_i$ is assigned with a reference point $w_k^i$.
At the $k$-th iteration, a random index $i_k$ is sampled from $[n]$, and the update $w_{k+1}^{i_k}=x_k$ is performed while other reference points remain unchanged.
Then the stochastic gradient $v_k$ is computed as
\begin{equation}\label{eq: saga}
	v_{k}=\nabla f_{i_k}(x_k)-\nabla f_{i_k}(w_{k}^{i_k})+\frac{1}{n}\sum_{j=1}^n\nabla f_j(w_{k}^j).
\end{equation}
Next, we show that $v_k$ defined by \eqref{eq: saga} satisfies Assumption \ref{ass: vr}, in which case $\e_k[\cdot]$ denotes the expectation conditioned on $x_k$ and $\{w_k^i\}_{i=1}^n$.
Similar results can be established for some other VR methods.

\begin{proposition}
	Consider problem \eqref{eq: main prob}. Suppose each component $f_i$ is convex and $L$-smooth, and $x^* \in \argmin_{x\in \rbb^d} F(x)$.
	Then, the stochastic gradient $v_k$ defined by \eqref{eq: saga} satisfies Assumption \ref{ass: vr} with
	$a=4 L, b=2, \gamma=\frac{1}{n}, c=\frac{2 L}{n}$ and
	\begin{equation*}
		\xi_k=\frac{1}{n} \sum_{i=1}^n\left\|\nabla f_i\left(w_k^i\right)-\nabla f_i\left(x^*\right)\right\|^2.
	\end{equation*}
\end{proposition}

\begin{proof}
	To start off, it is straightforward to verify $\e_k[v_k] = \nabla f(x_k)$. Now, we estimate the variance of the stochastic gradient $v_k$ defined by \eqref{eq: saga}.
	First, it follows from \eqref{eq: saga} that
	\begin{equation}\label{eq: saga-1}
		\begin{aligned}
			\e_k\left[\left\|v_k-\nabla f\left(x_k\right)\right\|^2\right] = &\e_k\left[\left\|\nabla f_{i_k}\left(x_k\right)-\nabla f_{i_k}\left(w_k^{i_k}\right)+\frac{1}{n} \sum_{j=1}^n \nabla f_j\left(w_k^j\right)-\nabla f\left(x_k\right)\right\|^2\right] \\
			= &\e_k\left[\left\|\nabla f_{i_k}\left(x_k\right)-\nabla f_{i_k}\left(w_k^{i_k}\right)-\e_k\left[\nabla f_{i_k}\left(x_k\right)-\nabla f_{i_k}\left(w_k^{i_k}\right)\right]\right\|^2\right] \\
			\leq & \e_k\left[\left\|\nabla f_{i_k}\left(x_k\right)-\nabla f_{i_k}\left(x^*\right)+\nabla f_{i_k}\left(x^*\right)-\nabla f_{i_k}\left(w_k^{i_k}\right)\right\|^2\right].
		\end{aligned}
	\end{equation}
	Apply Young's inequality on the RHS of \eqref{eq: saga-1} and then apply Lemma \ref{lem: pre_bv}, we have
	\begin{equation*}
		\begin{aligned}
			\e_k\left[\left\|v_k-\nabla f\left(x_k\right)\right\|^2\right] \leq & 2 \e_k\left[\left\|\nabla f_{i_k}\left(x_k\right)-\nabla f_{i_k}\left(x^*\right)\right\|^2\right]+2 \e_k\left[\left\|\nabla f_{i_k}\left(w_k^{i_k}\right)-\nabla f_{i_k}\left(x^*\right)\right\|^2\right] \\
			\leq & 4 L\left(F\left(x_k\right)-F^*\right)+\frac{2}{n} \sum_{i=1}^n\left\|\nabla f_i\left(w_k^i\right)-\nabla f_i\left(x^*\right)\right\|^2\\
            = & 4 L\left(F\left(x_k\right)-F^*\right)+2\xi_k,
		\end{aligned}
	\end{equation*}
	
	Next, we estimate $\e_k[\xi_{k+1}]$ by Lemma \ref{lem: pre_bv}, which gives
	\begin{equation*}
		\begin{aligned}
			\e_k\left[\xi_{k+1}\right] & =\frac{1}{n} \sum_{i=1}^n \e_k\left[\left\|\nabla f_i\left(w_{k+1}^i\right)-\nabla f_i\left(x^*\right)\right\|^2\right] \\
			& =\frac{1}{n} \sum_{i=1}^n\left(\frac{n-1}{n}\left\|\nabla f_i\left(w_k^i\right)-\nabla f_i\left(x^*\right)\right\|^2+\frac{1}{n}\left\|\nabla f_i\left(x_k\right)-\nabla f_i\left(x^*\right)\right\|^2\right) \\
			& \leq\left(1-\frac{1}{n}\right) \xi_k+\frac{2 L}{n}\left(F\left(x_k\right)-F^*\right).
		\end{aligned}
	\end{equation*}
	The proof is now complete.
\end{proof}

In the subsequent theorem, we extend Theorem \ref{thm: main} and demonstrate that a linear convergence rate can be achieved by generalizations of Algorithm \ref{algo: spbfgs}, wherein the stochastic gradient $v_k$ satisfies Assumption \ref{ass: vr}.

\begin{theorem}\label{thm: 2}
	Let Assumption \ref{ass: fh} hold and consider implementing a variant of Algorithm \ref{algo: spbfgs}, wherein the stochastic gradients $\{v_k\}$ are defined such that Assumption \ref{ass: vr} is satisfied.
	Suppose, we choose constant $A>0$ and step size $\eta>0$ such that
	\begin{equation*}
		\rho :=\max \left\{1-\frac{\eta \mu}{m_2}, \frac{2b\sqrt{m_2}}{A m_1\sqrt{m_1}}+1-\gamma\right\} \in (0, 1),~\left(\frac{2a\sqrt{m_2}}{m_1\sqrt{m_1}}+A c\right)\eta \leq 2 \rho,
	\end{equation*}
	where $m_1$ and $m_2$ are positive constants such that $m_1 I \preceq B_k \preceq m_2 I$ for any $k\geq 0$. Besides, the Hessian update frequency $r>0$ is chosen to satisfy $\big(\frac{m_2}{m_1}\big)^{1/r} \rho < 1$.
	Then for any $k\ge r$, we have
	\begin{equation*}
		\mathbb{E}\left[V_k\right] \leq\left(\frac{m_2}{m_1}\rho^r\right)^{\lfloor \frac{k}{r} \rfloor} V_{0},
	\end{equation*}
	where $V_k$ is a Lyapunov function defined by
	\begin{equation*}
		V_k=\left\|x_k-x^*\right\|_{B_k}^2+A\eta^2 \xi_k+2 \eta\left[F\left(x_k\right)-F^*\right].
	\end{equation*}
\end{theorem}

\begin{proof}
	Imitating the proof of Theorem \ref{thm: main} from \eqref{eq: key lem} to \eqref{eq: e_ip_delta} and applying the first inequality in \eqref{eq: ass_vr}, we have
	\begin{equation}\label{eq: x-expand}
		\begin{aligned}
			\e_k\left[\left\|x_{k+1}-x^*\right\|_{B_k}^2\right] \leq & \left(1-\frac{\eta\mu}{m_2}\right)\left\|x_k-x^*\right\|_{B_k}^2-2 \eta \e_k\left[F\left(x_{k+1}\right)-F^*\right]\\
			&+\frac{2\sqrt{m_2}}{m_1\sqrt{m_1}} \eta^2 \e_k\left[\left\|v_k - \nabla f(x_k)\right\|^2\right] \\
			\leq & \left(1-\frac{\eta\mu}{m_2}\right)\left\|x_k-x^*\right\|_{B_k}^2-2 \eta \e_k\left[F\left(x_{k+1}\right)-F^*\right]\\
			&+\frac{2a\sqrt{m_2}}{m_1\sqrt{m_1}} \eta^2\left(F\left(x_k\right)-F^*\right)+\frac{2b\sqrt{m_2}}{m_1\sqrt{m_1}} \eta^2 \xi_k.
		\end{aligned}
	\end{equation}
	Adding $A\eta^2 \e_k[\xi_{k+1}]$ on both sides of \eqref{eq: x-expand} and applying the second inequality in \eqref{eq: ass_vr} leads to
	\begin{equation*}
		\begin{aligned}
			& \e_k \left[\| x_{k+1}-x^*\|_{B_k}^2+A \eta^2 \xi_{k+1}+2 \eta \left[F\left(x_{k+1}\right)-F^*\right]\right]\\
			\leq & \left(1-\frac{\eta\mu}{m_2}\right)\left\|x_k-x^*\right\|_{B_k}^2+\left(\frac{2b\sqrt{m_2}}{A m_1\sqrt{m_1}}+1-\gamma\right) A \eta^2 \xi_k+\left(\frac{2a\sqrt{m_2}}{m_1\sqrt{m_1}} +A c\right) \eta^2\left[F\left(x_k\right)-F^*\right].
		\end{aligned}
	\end{equation*}
	Since $\rho=\max \left\{1-\frac{\eta\mu}{m_2}, \frac{2b\sqrt{m_2}}{A m_1\sqrt{m_1}}+1-\gamma\right\}\in (0,1)$, we have
	\begin{equation*}
		\begin{aligned}
			& \e_k\left[\left\|x_{k+1}-x^*\right\|_{B_k}^2+A \eta^2 \xi_{k+1}+2 \eta \left[F\left(x_{k+1}\right)-F^*\right]\right] \\
			\leq & \rho\left[\left\|x_k-x^*\right\|_{B_k}^2+A \eta^2 \xi_k+\frac{\eta^2}{\rho}\left(\frac{2a\sqrt{m_2}}{m_1\sqrt{m_1}}+A c\right)\left(F\left(x_k\right)-F^*\right)\right].
		\end{aligned}
	\end{equation*}
	Since $\left(\frac{2a\sqrt{m_2}}{m_1\sqrt{m_1}}+A c\right)\eta \leq 2 \rho$, it holds that $\frac{\eta^2}{\rho}\left(\frac{2a\sqrt{m_2}}{m_1\sqrt{m_1}}+A c\right) \leq 2 \eta$.
	Therefore, we have
	\begin{equation*}
		\begin{aligned}
			& \mathbb{E}_k\left[\left\|x_{k+1}-x^*\right\|_{B_k}^2+A\eta^2 \xi_{k+1}+2 \eta\left(F\left(x_{k+1}\right)-F^*\right)\right] \\
			\leq & \rho\left[\left\|x_k-x^*\right\|_{B_k}^2+A\eta^2 \xi_k+2\eta\left(F\left(x_k\right)-F^*\right)\right].
		\end{aligned}
	\end{equation*}
	The remaining of the proof resembles the proof of Theorem \ref{thm: main} starting from \eqref{eq: thm1-1}.
\end{proof}

\begin{remark}
	Similar to Theorem \ref{thm: plsvrg}, we can establish convergence results for the proximal extensions of various VR methods satisfying Assumption \ref{ass: vr} by setting $m_1=m_2=1$ in Theorem \ref{thm: 2}.
\end{remark}

\section{An SSN method for (\ref{eq:subproblem}) and its numerical implementation}
\label{sec: ssn}
In this section, we provide a fast SSN method for efficiently solving the subproblem \eqref{eq:subproblem}.
The subproblem \eqref{eq:subproblem} is equivalent to the following minimization problem:
\begin{equation} \label{eq: sub_ori}
	\min_{x\in\rbb^d} \quad v_k^\top\left(x-x_k\right)+\frac{1}{2 \eta_k}\left(x-x_k\right)^\top B_k (x-x_k)+h(x).
\end{equation}
To simplify the formulation of \eqref{eq: sub_ori}, we introduce
\begin{equation*}
	g:= \eta_k v_k - B_k x_k, \quad\theta(\cdot) := \eta_k h(\cdot).
\end{equation*}
Then, we omit the subscript $k$ in $B_k$ and reformulate \eqref{eq: sub_ori} as
\begin{equation} \label{eq: sub_sim}
	\min_{x \in \rbb^d} \quad g^{\top} x + \frac{1}{2} x^{\top} B x + \theta(x) .
\end{equation}
Let $\{(s_i, y_i)\}_{i=1}^{m}$ be the correction pairs that generate $B$, where $m$ is the current memory size.
For simplicity and with a slight abuse of notation, we denote by $B_i$ the $i$-th matrix in the generation of $B$ throughout this section. That is, given $B_{i-1}$, $B_i$ is updated via
\begin{equation*}
	B_i = B_{i-1} - \frac{s_i y_i^\top}{y_i^\top s_i} + \frac{y_i y_i^\top}{y_i^\top s_i},
\end{equation*} 
and we have $B = B_m$.

In Section \ref{sec: ssn-discription}, we present our SSN method for solving \eqref{eq: sub_sim}, and we provide an efficient numerical implementation for the proposed SSN method in Section \ref{sec: ssn-implement}.

\subsection{An SSN method for (\ref{eq: sub_sim})}\label{sec: ssn-discription}
It is known that SSN methods are typically used in conjunction with a suitable step size search strategy \cite{ali2017semismooth, nakayama2024inexact}.
However, the nonsmoothness of the objective function in \eqref{eq: sub_sim} prevents the direct application of commonly used line search conditions. To address this issue, we first transform problem \eqref{eq: sub_sim} into an unconstrained smooth problem.
For this purpose, we first rewrite \eqref{eq: sub_sim} as
\begin{equation} \label{model_v1}
	\begin{aligned}
		\min_{x, z \in \rbb^d} \quad& g^{\top} x+ \frac{1}{2} x^{\top}(B-\alpha I) x+\frac{\alpha}{2} z^{\top} z+ \theta(z), \\
		\text { s.t. }\; & x=z,
	\end{aligned}
\end{equation}
where $\alpha>0$.
Note that if the matrix $B_\alpha:=B-\alpha I$ is positive definite, then the objective function of \eqref{model_v1} is convex.
The proposition below offers guidance for selecting $\alpha$ to ensure the positive definiteness of $B_\alpha$.

\begin{proposition} \label{prop: choice_alpha}
	For any $\alpha$ satisfying $0 \le \alpha<\bar{\alpha}:= \frac{1}{\frac{1}{\sigma_0}+\sum_{i=1}^{m} \frac{s_{i}^\top s_i}{y_{i}^\top s_i}}$, where $\sigma_0$ is defined in \eqref{eq: compact_lbfgs}, the matrix $B_\alpha=B-\alpha I$ is positive definite.
\end{proposition}

\begin{proof}
	Recall $y_i^\top s_i>0$ for all $i \in [m]$. From \cite{nocedal1999numerical}, the inverse L-BFGS matrix $H_{i} = B_{i}^{-1}$ satisfies for any $i \in [m]$,
	\[
	H_{i}=\Big(I-\frac{s_i y_i^\top}{y_i^\top s_i}\Big)H_{i-1}\Big(I-\frac{y_i s_i^\top}{y_i^\top s_i}\Big)+\frac{s_i s_i^\top}{y_i^\top s_i}.
	\]
	Note that the largest eigenvalue of $H_i$ denoted by $\sigma_{\max} (H_{i})$ satisfies
		\begin{equation}\label{eq: alpha-1}
		\begin{aligned}
			\sigma_{\max} (H_{i}) &=\max_{\|x\|=1}x^\top H_{i} x\\
			&\le \max_{\|x\|=1}x^\top \Big(I-\frac{s_i y_i^\top}{y_i^\top s_i}\Big)H_{i-1}\Big(I-\frac{y_i s_i^\top}{y_i^\top s_i}\Big) x+\max_{\|x\|=1}x^\top \frac{s_i s_i^\top}{y_i^\top s_i}  x\\
			&=\sigma_{\max} \Big(\Big(I-\frac{s_i y_i^\top}{y_i^\top s_i}\Big)H_{i-1}\Big(I-\frac{y_i s_i^\top}{y_i^\top s_i}\Big) \Big)+\frac{s_i^\top s_i}{y_i^\top s_i}.
			\end{aligned}
		\end{equation}
	On the other hand, note that the eigenvalues of matrix $\frac{y_i s_i^\top}{y_i^\top s_i}$ correspond to 0 with multiplicity $d - 1$ and 1 with multiplicity 1, hence there exist orthogonal matrices $U_i,V_i\in \mathbb{R}^{d\times d}$ and diagonal matrix $\Sigma=\operatorname{diag}(0, 1, \dots, 1)\in \mathbb{R}^{d\times d}$ such that
		\begin{equation*}
	I-\frac{y_i s_i^\top}{y_i^\top s_i}=U_i \Sigma V_i.
		\end{equation*}
	Thus it holds that
	\begin{equation}\label{eq: alpha-2}
		\begin{aligned}
			\sigma_{\max} \Big(\Big(I-\frac{s_i y_i^\top}{y_i^\top s_i}\Big)H_{i-1}\Big(I-\frac{y_i s_i^\top}{y_i^\top s_i}\Big) \Big)&=\sigma_{\max} (V_i^\top \Sigma U_i^\top H_{i-1}U_i \Sigma V_i)=\sigma_{\max} ( \Sigma U_i^\top H_{i-1}U_i \Sigma )\\
			&\le \sigma_{\max} (U_i^\top H_{i-1}U_i)=\sigma_{\max} (H_{i-1}),
			\end{aligned}
		\end{equation}
	where the inequality comes from the interlacing property for eigenvalues of Hermitian matrices (see \cite{golub2013matrix}) by noting that the nonzero $(d-1)\times (d-1)$ principle submatrix of $\Sigma U_i^\top H_{i-1}U_i \Sigma$ is also a principle submatrix of $U_i^\top H_{i-1}U_i$.
	Then from \eqref{eq: alpha-1} and \eqref{eq: alpha-2} we have
	\begin{equation}\label{eq: alpha-3}
	\sigma_{\max} (H_{i})\le  \sigma_{\max} ( H_{i-1}  )+\frac{s_i^\top s_i}{y_i^\top s_i}.
		\end{equation}
		Summing over \eqref{eq: alpha-3} for $i=1,\dots,m$ gives
		\begin{equation*}
			\sigma_{\max}(H_{m}) \leq \sigma_{\max}(H_{0})+\sum_{i=1}^{m} \frac{s_{i}^\top s_i}{y_{i}^\top s_i}=\frac{1}{\sigma_{0}}+\sum_{i=1}^{m} \frac{s_{i}^\top s_i}{y_{i}^\top s_i}.
		\end{equation*}
		Denote by $\sigma_{\min} (B_{m})$ the smallest eigenvalue of $B_{m}$, then we have
		\begin{equation*}
			\sigma_{\min} (B_{m}) =\frac{1}{\sigma_{\max} (H_{m})} \ge \frac{1}{\frac{1}{\sigma_0} +\sum_{i=1}^{m} \frac{s_{i}^\top s_i}{y_{i}^\top s_i}} := \bar{\alpha},
		\end{equation*}
		which together with $B = B_m$ implies that $\sigma_{\min} (B_{\alpha})> 0$ for all $\alpha \in [0, \bar{\alpha})$, hence completes the proof.
	\end{proof}

To proceed, let's consider the dual problem of \eqref{model_v1}, which reads as
\begin{equation} \label{eq: DP_ori}
	\max_{\lambda \in \mathbb{R}^d} J^*(\lambda),
\end{equation}
where $J^*(\lambda)= \min_{x,z\in\rbb^d} L(x, z;\lambda)$ with
$L(x, z;\lambda)= \frac{1}{2}x^\top B_{\alpha} x+g^\top x+ \frac{\alpha}{2} z^\top z+\theta(z)-\lambda^\top(x-z)$.
Since problem \eqref{model_v1} is convex with $0\le \alpha<\bar{\alpha}$, the strong duality holds \cite{boyd2004convex}. Thus, the solutions of \eqref{model_v1} can be obtained by solving its dual problem \eqref{eq: DP_ori}.
For convenience, we define $\Lambda(\lambda):= - J^*(\lambda)$ and consider the following equivalent formulation of \eqref{eq: DP_ori}:
\begin{equation}\label{eq: DP}
	\min_{\lambda \in \mathbb{R}^d} \Lambda(\lambda).
\end{equation}

\begin{remark}
	
		We transform the subproblem \eqref{eq: sub_sim} into the dual reformulation \eqref{eq: DP}, which can be conducted for general positive definite matrix $B$. It is worth noting that an alternative reformulation of \eqref{eq: sub_sim} is provided by the proximal calculus framework in Becker et al. \cite{becker2019quasi}. This reformulation is applicable when the matrix $B$ admits a decomposition $B=D - UU^\top$, where $D\in\rbb^{d\times d}$ is a diagonal matrix and $U\in\rbb^{d\times m}$. In this case, the solution of \eqref{eq: sub_sim}, expressed as the scaled proximal operator $\operatorname{prox}_\theta^B(u)$ with $u:=-B^{-1}g$, satisfies
		\[
		\p_\theta^B(u) =\p_\theta^D\left(u + D^{-1} U \alpha^{\star}\right),
		\]
		where $\alpha^{\star}$ is the unique root of $\mathcal{L}(\alpha) :=U^{\top}\left(u-\p_\theta^D\left(u + D^{-1} U \alpha\right)\right)+\alpha$. Consequently, the computational challenge shifts to solving the nonlinear, nonsmooth equation $\mathcal{L}(\alpha) = 0$, which typically requires an iterative solver such as an SSN method. However, for L-BFGS matrices with a general memory size $m>1$, obtaining the required decomposition $B=D - UU^\top$ is not straightforward. As seen from the compact representation \eqref{eq: compact_lbfgs}, this decomposition necessitates computing the square root of the inverse of a matrix $J$, which is one of the key computational bottlenecks that our method avoids (see Section \ref{sec: ssn-implement}).
	
\end{remark}

Next, we show that the objective function $\Lambda$ in (\ref{eq: DP}) is differentiable if $0<\alpha<\bar{\alpha}$.

\begin{proposition} \label{prop: gradient_Lambda}
	If $0<\alpha<\bar{\alpha}$, then the objective function $\Lambda$ in \eqref{eq: DP} is differentiable and its gradient is given by
	\begin{equation*}
		\nabla \Lambda(\lambda)=  B_\alpha^{-1}(\lambda-g)-\p_{\frac{1}{\alpha},\;\theta}\Big(-\frac{\lambda}{\alpha}\Big),~\forall \lambda\in\rbb^d.
	\end{equation*}
\end{proposition}

\begin{proof}
	To simplify the notation, we introduce two auxiliary functions $\Psi(\cdot)$ and $\Theta(\cdot)$ defined by
	\begin{equation*}
		\begin{aligned}
			\Psi(\lambda) &:= -\min_{x\in\rbb^d}\;\frac{1}{2}x^\top  B_{\alpha} x+g^\top x-\lambda^\top x=\frac{1}{2}(\lambda-g)^\top  B_\alpha^{-1}(\lambda-g),\\
			\Theta(\lambda) &:=-\min_{z\in\rbb^d} \; \frac{\alpha}{2}\|z\|^2+\theta (z)+\lambda^\top z.
		\end{aligned}
	\end{equation*}
	Then, it is easy to see that
	\begin{equation*}
		\Lambda(\lambda)=\Psi(\lambda)+\Theta(\lambda).
	\end{equation*}
	Therefore, it suffices to prove that both $\Psi$ and $\Theta$ are differentiable.
	
	To this end, first note that $\Psi$ is differentiable with its gradient given by
	\begin{equation}\label{diff_psi}
		\nabla \Psi(\lambda)=  B_\alpha^{-1}(\lambda-g),\;\forall \lambda\in\rbb^d.
	\end{equation}
	To show that $\Theta$ is also differentiable, notice that
	\begin{equation*}
		\Theta(\lambda)=-\min_{z\in\rbb^d} \left\{\theta(z)+\frac{\alpha}{2}\|z+\frac{\lambda}{\alpha}\|^2-\frac{1}{2\alpha}\|\lambda\|^2\right\}=-\mathcal{M}_{\frac{1}{\alpha},\;\theta}\Big(-\frac{\lambda}{\alpha}\Big)+\frac{1}{2\alpha}\|\lambda\|^2.
	\end{equation*}
	Hence, from \eqref{eq: diff_M} we know that $\Theta$ is differentiable, and
	\begin{equation}\label{diff^theta}
		\nabla \Theta(\lambda)=\frac{1}{\alpha}\nabla\mathcal{M}_{\frac{1}{\alpha},\;\theta}\Big(-\frac{\lambda}{\alpha}\Big)+\frac{\lambda}{\alpha}=-\p_{\frac{1}{\alpha},\;\theta}\Big(-\frac{\lambda}{\alpha}\Big).
	\end{equation}
	Consequently, $\Lambda$ is differentiable, and
	combining (\ref{diff_psi}) with (\ref{diff^theta}) gives
	\[
		\nabla \Lambda(\lambda)=  B_\alpha^{-1}(\lambda-g)-\p_{\frac{1}{\alpha},\;\theta}\Big(-\frac{\lambda}{\alpha}\Big),
	\]
	which completes the proof.
\end{proof}

As a consequence of Proposition \ref{prop: gradient_Lambda}, if we let $\mathcal{D}_{\theta}(\cdot)$ be the generalized Jacobian \cite{clarke1990optimization} of $\p_\theta(\cdot)$, then the matrix
\begin{equation*}
	B_\alpha^{-1}+\frac{1}{\alpha}\mathcal{D}_{\frac{ \theta}{\alpha}}\Big(-\frac{\lambda}{\alpha}\Big)
\end{equation*}
is the generalized Jacobian of $\nabla \Lambda $ at $\lambda$. Now applying an SSN method to \eqref{eq: DP} readily yields the following iterative scheme
\begin{equation*}
\lambda_{j+1}=\lambda_j+\rho_jd_j,
\end{equation*}
where $\rho_j>0$ is a step size and $d_j:=-( B_\alpha^{-1}+\mathcal{D}_j)^{-1}\nabla \Lambda(\lambda_j)$ with $\mathcal{D}_j:=\frac{1}{\alpha}\mathcal{D}_{\frac{ \theta}{\alpha}}\big(-\frac{\lambda_j}{\alpha}\big)$. 

To determine the step size $\rho_j$, one can apply a backtracking scheme until the Wolfe condition \cite{wolfe1969convergence} is met.
Alternatively, we propose an exact line search by solving
 $$\rho_j=\arg\min_{\rho\ge 0} R_j(\rho):=\Lambda(\lambda_j+\rho d_j).$$
Since $R_j$ is convex and differentiable, various methods can be applied to minimize it.
In particular, we advocate applying an SSN method due to its fast convergence.
For this purpose, a straightforward calculation gives
\begin{equation*}
	R^{'}_j(\rho)=d_j^\top \nabla \Lambda(\lambda_j+\rho d_j)=d_j^\top B_\alpha^{-1}(\lambda_j-g)+\rho d_j^\top B_\alpha^{-1} d_j - d_j^\top \p_{\frac{1}{\alpha},\;\theta}\Big(- \frac{\lambda_j+\rho d_j}{\alpha}\Big),
\end{equation*}
which implies that
\begin{equation*}
	d_j^\top B_{\alpha}^{-1} d_j +\frac{1}{\alpha} d_j^\top  \mathcal{D}_{\frac{ \theta}{\alpha}}\Big(-\frac{\lambda_j+\rho d_j}{\alpha}\Big)d_j
\end{equation*}
is an element of the generalized Jacobian of $R^{'}_j(\rho)$.
Therefore, an SSN method for minimizing $R_j(\cdot)$ proceeds by updating
\begin{equation}\label{SSN_LS}
	\rho_{i+1,j}=\rho_{i,j}-\frac{d_j^\top B_\alpha^{-1}(\lambda_j-g)+\rho_{i,j} d_j^\top B_{\alpha}^{-1} d_j - d_j^\top \p_{\frac{1}{\alpha},\;\theta}\big(- \frac{\lambda_j+\rho_{i,j} d_j}{\alpha}\big)}{d_j^\top B_{\alpha}^{-1} d_j +\frac{1}{\alpha} d_j^\top  \mathcal{D}_{\frac{ \theta}{\alpha}}\big(-\frac{\lambda_j+\rho_{i,j}  d_j}{\alpha}\big)d_j}.
\end{equation}
Note that $\rho_{0,j}=1$ is a natural choice of the initial guess, and we set $\rho_j$ as the output of \eqref{SSN_LS}.

With the above discussions, we propose an SSN method for solving problem \eqref{eq: DP} and list it in Algorithm \ref{algo: ssn}.

\begin{algorithm}
	\caption{An SSN method for solving problem \eqref{eq: DP}.}
	\label{algo: ssn}
	\begin{algorithmic}
		\Require initial point $\lambda_0\in \rbb^d$.
		\For{$j = 0, 1, \dots$}
		\State Compute $x_{j}= B_\alpha^{-1}(\lambda_j-g)$ and $z_{j}=\p_{\frac{1}{\alpha},\;\theta}\big(- \frac{\lambda_j}{\alpha}\big)$
		\State Set $\mathcal{D}_j=\frac{1}{\alpha}\mathcal{D}_{\frac{ \theta}{\alpha}}\big(-\frac{\lambda_j}{\alpha}\big)$ and $\nabla \Lambda(\lambda_j)=x_{j}-z_{j}$
		\State Compute $d_j=-( B_\alpha^{-1}+\mathcal{D}_j)^{-1}\nabla \Lambda(\lambda_j)$
		\State Update $\lambda_{j+1}=\lambda_j+\rho_j d_j$ with the step size $\rho_j$ obtained by (\ref{SSN_LS}).
		\EndFor
	\end{algorithmic}
\end{algorithm}

\begin{remark}
	From the equivalence between problems \eqref{eq: sub_sim} and \eqref{model_v1} and the fact that strong duality holds between \eqref{model_v1} and its dual \eqref{eq: DP}, we have that $x^*= B_\alpha^{-1}(\lambda^*-g)$ is a global minimizer of \eqref{eq: sub_sim} if $\lambda^*$ is a global minimizer of \eqref{eq: DP}.
	In other words, Algorithm \ref{algo: ssn} is an SSN method from a dual perspective to solve \eqref{eq: sub_sim}.
\end{remark}

\subsection{A fast numerical implementation of Algorithm \ref{algo: ssn}}
\label{sec: ssn-implement}
In this subsection, we explore an efficient numerical implementation of Algorithm \ref{algo: ssn}. Assuming that $\theta$ has a simple structure with an efficiently computable proximal operator, the computational cost of Algorithm \ref{algo: ssn} primarily involves computing $x_j$ and $d_j$, as well as the computations involved in \eqref{SSN_LS}. As a result, the key challenge in implementing Algorithm \ref{algo: ssn} is:
\textit{given a vector $z\in \mathbb{R}^d$, how to efficiently compute the matrix-vector products $B_\alpha^{-1}z$ and $( B_\alpha^{-1}+\mathcal{D}_j)^{-1}z$.}

Recall that at most $l\ll d$ correction pairs $\left(s_j, y_j\right), j=1,\ldots, m\leq l$, are stored.
From \eqref{eq: compact_lbfgs}, we can represent $B_\alpha$ as a scaled difference between a diagonal matrix and a matrix with rank $2m\le 2l\ll d$, i.e.,
\begin{equation}\label{B_alpha}
	B_\alpha=(\sigma_{0} -\alpha)(I - UJ^{-1}U^\top),
\end{equation}
where
\begin{equation*}
	U:=\begin{bmatrix}\sigma_{0} S & Y\end{bmatrix}\in \mathbb{R}^{d\times 2m }\text{~and~}J := \left(\sigma_{0} -\alpha\right)\begin{bmatrix} \sigma_{0} S^\top  S & L\\ L^\top & -D\end{bmatrix}\in \mathbb{R}^{2m \times 2m }.
\end{equation*}
Leveraging this specific structure, we can convert the complex $d$-dimensional linear system in Algorithm \ref{algo: ssn} into an easily solvable
$2m$-dimensional linear system.

For simplicity, we assume that $ \theta = \|\cdot\|_1$ or $ \theta=I_C$, with $I_C$ being the indicator function of a convex closed set $C\subset \rbb^d$.
In both cases, the generalized Jacobian $\mathcal{D}_{\theta}(\cdot)$ of $\p_{\theta}(\cdot)$ (see \cite{clarke1990optimization}) can be represented as a diagonal matrix with entries being either $0$ or $1$. This distinct structure of $\mathcal{D}_{\theta}(\cdot)$ implies that many of the multiplication operations employed in computing
 $( B_\alpha^{-1}+\mathcal{D}_j)^{-1}z$ are the same as those used in computing $B_\alpha^{-1}z$.
Then as shown below, we can reduce the redundant multiplication operations by computing and storing some auxiliary matrices. This approach significantly reduces the computational load of the entire algorithm.
We summarize these conditions in Assumption \ref{ass: assumption2}. It is worth mentioning that the techniques developed subsequently can be easily generalized to cases where $\mathcal{D}_{\theta}$ and the basic matrix $B_{0}$ are general diagonal matrices.
\begin{assumption} \label{ass: assumption2}
	For any $u\in\rbb^d$, the generalized Jacobian $\mathcal{D}_{\theta}(u)$ is a diagonal matrix with diagonal entries being either $0$ or $1$.
\end{assumption}

\subsubsection{Computation of auxiliary matrices}

In addition to the matrices $S=[s_{1}, \dots, s_{m}]$ and $Y=[y_{1}, \dots, y_{m}]$, we introduce some supplementary auxiliary matrices to track all multiplicative operations required for computing $S^\top Y$, $S^\top S$, and $Y^\top Y$. To simplify our notation, for any vectors $u,v\in \mathbb{R}^m$ we define the operation:
\begin{equation*}
	u\otimes v:=u v^\top=\begin{bmatrix}u_{1 }v_{1 }&u_{1 }v_{2 } &\hdots &u_{1 }v_{m } \\
	u_{2 }v_{1 }&u_{2 }v_{2 } &\hdots &u_{2 }v_{m } \\
	\vdots &\vdots &\ddots  &\vdots\\
	u_{m }v_{1 }&u_{m  }v_{2 } &\hdots &u_{m }v_{m } \\
	   \end{bmatrix}\in\rbb^{m\times m}.
\end{equation*}
For $i\in [d]$, let $\bar{s}_i, \bar{y}_i \in \rbb^m$ represent the transpose of the $i$-th row of $S$ and $Y$, respectively, such that:
\begin{equation*}
	S=\begin{bmatrix}\bar{s}_1^\top\\\bar{s}_2^\top\\ \vdots\\ \bar{s}_d^\top\end{bmatrix}\in\rbb^{d\times m} \text{~and~} Y=\begin{bmatrix}\bar{y}_1^\top\\\bar{y}_2^\top\\ \vdots\\ \bar{y}_d^\top\end{bmatrix}\in\rbb^{d\times m}.
\end{equation*}
We compute and store the following auxiliary matrices:
\begin{equation*}
	\bar{s}_i \otimes \bar{y}_i,~\bar{s}_i \otimes \bar{s}_i ,~\text{and}~\bar{y}_i \otimes \bar{y}_i\in \mathbb{R}^{m\times m} ,~\text{for}~i=1, 2, \dots, d.
\end{equation*}
After computing a new correction pair $(s^t, y^t)$, these auxiliary matrices do not need to be completely recomputed. In fact, for $i\in [d]$, let $s_{j, i}, s_i^t\in\rbb$ denote the $i$-th entry of $s_j$ and $s^t$, respectively ($y_{j, i}, y_i^t\in \rbb$ are defined similarly). Then the updates only include computing the entries:
\begin{equation}\label{entries}
	\left\{
	\begin{aligned}
		&s_{j,i}~y_{ i}^t \\
		&s_{i}^t~y_{j, i} \\
		&s_{j, i}~s_{i}^t \\
		&y_{j, i}~y_{i}^t \\
		\end{aligned}\right. \text{~for~}j\in \left\{
			\begin{aligned}
				&[m]\text{~if~} t\leq l\\
				&\{2, \dots, m\}\text{~otherwise}  \\
			\end{aligned}\right.\text{~and~}	\left\{
	\begin{aligned}
		&s_{i}^t~y_{i}^t,\\
		&s_{i}^t~s_{i}^t,  \\
		&y_{i}^t~y_{i}^t.  \\
	\end{aligned}\right.
\end{equation}
When $t\le l$, we update $S$ and $Y$ by adding new columns $s_{m+1}=s^t$ and $y_{m+1}=y^t$ respectively, and update $\bar{s}_i \otimes \bar{y}_i,~\bar{s}_i \otimes \bar{s}_i,~\text{and}~\bar{y}_i \otimes \bar{y}_i$ by appending a new column and row with entries calculated in \eqref{entries} on the right and bottom correspondingly.
When $t>l$, the matrices $S$, $Y$ and the auxiliary matrices $\bar{s}_i \otimes \bar{y}_i,~\bar{s}_i \otimes \bar{s}_i,~\text{and}~\bar{y}_i \otimes \bar{y}_i$ are updated by deleting and appending certain rows and columns with $s^t$, $y^t$, and those computed in \eqref{entries}.
In total, computing and storing the auxiliary matrices $\bar{s}_i \otimes \bar{y}_i,~\bar{s}_i \otimes \bar{s}_i,~\text{and}~\bar{y}_i \otimes \bar{y}_i$ for $i\in [d]$ totally require approximately $4md$ multiplication operations and $2m^2d$ units of storage.

\subsubsection{Computation of \texorpdfstring{$B_\alpha^{-1} z$}{}}\label{sec: two}
Recall \eqref{B_alpha}, we have from Woodbury's formula that
 \begin{equation*}
	 B_{\alpha}^{-1}z=\frac{1}{\sigma_{0} -\alpha}[ z-U (U^\top U-J)^{-1}U^\top z].
\end{equation*}
Therefore, $B_\alpha^{-1} z$ can be easily computed once we obtain the explicit expression of $(U^\top U-J)^{-1}$.
First, it holds that
$$
J = \left(\sigma_{0} -\alpha\right)\begin{bmatrix} \sigma_{0} S^\top  S & L\\ L^\top & -D\end{bmatrix}\in \mathbb{R}^{2m \times 2m }\text{~and~} U^\top U= \begin{bmatrix} \sigma_{0}^2 S^\top  S & \sigma_{0} S^\top Y \\  \sigma_{0} Y^\top  S & Y^\top  Y \end{bmatrix}\in \mathbb{R}^{2m \times 2m }.
$$
Additionally, note that
\begin{equation*}
S^\top Y=\sum_{i=1}^d \bar{s}_i \otimes \bar{y}_i,~S^\top S=\sum_{i=1}^d \bar{s}_i \otimes \bar{s}_i,~Y^\top Y=\sum_{i=1}^d \bar{y}_i \otimes \bar{y}_i,
\end{equation*}
which only involves assembling the auxiliary matrices  $ \bar{s}_i \otimes \bar{y}_i$, $\bar{s}_i \otimes \bar{s}_i$, and $\bar{y}_i \otimes \bar{y}_i$, $i=1,\ldots, d$.
Therefore, the matrices $S^\top Y$, $S^\top S$, and $Y^\top Y$ can be very efficiently computed since they do not require multiplications but only approximately $3 m^2 d$ additions in total.
Moreover, recall that $L$ is the lower triangular part of $S^\top Y$ and $D$ is the diagonal matrix of $S^\top Y$, hence obtaining the expressions of $L$ and $D$ is straightforward.
Therefore, the explicit expression of $U^\top U-J$ can be obtained correspondingly, and it costs $O(m^3)$ operations to compute its inverse, which is negligible since $m\ll d$. Then the remaining computations for $B_{\alpha}^{-1}z$ mainly consist of matrix-vector products, which cost about $4md$ multiplications and additions.

\subsubsection{Computation of \texorpdfstring{$( B_{\alpha}^{-1}+\mathcal{D}_j)^{-1}z$}{}}\label{sec: three}
From the compact formulation \eqref{eq: compact_lbfgs}, we have from Woodbury's formula that
\begin{equation}\label{eq: 3rd-computation}
	\begin{aligned}
		(B_{\alpha}^{-1}+\mathcal{D}_j)^{-1}z =&\left(\frac{1}{\sigma_{0} -\alpha} I +\mathcal{D}_j -\frac{1}{\sigma_{0} -\alpha} U (U^\top U-J)^{-1}U^\top \right)^{-1}z\\
		=&C_{\alpha,j}z-  C_{\alpha,j} U\left[U^\top C_{\alpha,j} U - (\sigma_{0} -\alpha)(U^\top U -J)\right]^{-1}U^\top C_{\alpha,j} z,
	\end{aligned}
\end{equation}
where we define $C_{\alpha,j}=\left(\frac{1}{\sigma_{0} -\alpha} I +\mathcal{D}_j \right)^{-1}\in \mathbb{R}^{d\times d}$, with $\mathcal{D}_j:=\frac{1}{\alpha}\mathcal{D}_{\frac{ \theta}{\alpha}}\big(-\frac{\lambda_j}{\alpha}\big)$.
Recall that from Assumption \ref{ass: assumption2} we have $ \mathcal{D}_j=\frac{1}{\alpha}\text{diag}(a_1^j,a_2^j,\dots, a_d^j)$, where $a_i^j \in \{0,1\}$ for any $i\in[d]$.
Therefore, the top-right block of $U^\top  C_{\alpha,j} U$ reads as
\begin{equation}\label{eq: K-bar-topright}
	 \sigma_{0}  S^\top C_{\alpha,j}Y =\sum_{i=1}^d \frac{\sigma_0}{1/(\sigma_0-\alpha)+a_i^j/\alpha} \bar{s}_i \otimes \bar{y}_i.
\end{equation}
Furthermore, denoting $I_j=\{1\le i\le d|a_i^j=1\}$, then it follows from \eqref{eq: K-bar-topright} that
\begin{equation*}
	\sigma_{0}  S^\top C_{\alpha,j}Y = {\alpha}{(\sigma_0 - \alpha)}  \sum_{i\in I_j} \bar{s}_i \otimes \bar{y}_i+ {\sigma_0}{(\sigma_0 - \alpha)}  \left(S^\top Y- \sum_{i\in I_j} \bar{s}_i \otimes \bar{y}_i \right).
\end{equation*}
Similar computation can be established for other blocks in $U^\top  C_{\alpha,j} U$, which implies that the explicit expression of $U^\top  C_{\alpha,j} U$ can be computed efficiently via assembling the auxiliary matrices, which only require approximately $3m^2 |I_j|$ additions.
Given $J$, $U^\top U$ and $U^\top  C_{\alpha,j} U$, it costs $O(m^3)$ operations to compute $\left[U^\top  C_{\alpha,j} U - (\sigma_{0} -\alpha)(U^\top U -J)\right]^{-1}$.
Since $C_{\alpha,k}$ is diagonal and $m\ll d$, we obtain that the remaining computation of $(B_{\alpha}^{-1}+\mathcal{D}_j)^{-1}z$ by \eqref{eq: 3rd-computation} requires about $4 m d$ multiplications and additions.

 \subsection{Computational complexity analysis}
 Based on the implementation described above, we can significantly reduce the computational complexity of Algorithm \ref{algo: ssn}. Overall, the computational cost of our implementation includes:
 	\begin{itemize}
		\item updating supplementary auxiliary matrices $\bar{s}_i \otimes \bar{y}_i,~\bar{s}_i \otimes \bar{s}_i,~\text{and}~\bar{y}_i \otimes \bar{y}_i,~\text{for}~i\in [d]$ ;
		\item computing the gradient $\nabla \Lambda(\lambda_j)$;
		\item computing the direction $d_j=-( B_\alpha^{-1}+\mathcal{D}_j)^{-1}\nabla \Lambda(\lambda_j)$;
		\item performing step size searches  with (\ref{SSN_LS}).		
	\end{itemize}

Assuming an implementation of Algorithm \ref{algo: ssn} involves $\iota $ SSN iterations (typically a single-digit number), we now estimate the computational cost of the aforementioned tasks.

Firstly, updating the auxiliary matrices incurs relatively minor computational cost. These auxiliary matrices are updated only after receiving a new correction pair, occurring every $r$ iterations, with each update costing approximately $4md$ multiplications. Thus, on average, the computational cost for updating these auxiliary matrices in each implementation of Algorithm \ref{algo: ssn} is $4md/r$ multiplications.

Secondly, recall that the gradient $\nabla \Lambda(\lambda_j)=x_j-z_j$. As previously mentioned,  computing $x_j = B_\alpha^{-1}(\lambda_j-g)$ requires roughly $3m^2d$ additions to calculate $U^\top U-J$, $O(m^3)$ operations to compute $(U^\top U-J)^{-1}$, and $4md$ multiplications and additions for matrix-vector multiplication.  Since we assume $m \le l \ll d$, the computational cost of $O(m^3)$ operations compared to $3m^2d$ additions can be neglected. Furthermore,  the matrix $(U^\top U-J)^{-1}$ is only recomputed during correction pair updates and the matrix-vector multiplication is executed per iteration of SSN. Therefore, on average, computing $x_j = B_\alpha^{-1}(\lambda_j-g)$ for each implementation of Algorithm \ref{algo: ssn} cost approximately  $4\iota md$ multiplications and $4\iota md + 3m^2d/r$ additions. On the other hand, computing $z_j $ involves $\iota  d$ non-linear operations in each implementation of Algorithm \ref{algo: ssn}, with the specific complexity depending on the form of the function $\theta$.  Under Assumption \ref{ass: assumption2}, these operations are easy and cheap to implement. That is to say, on average, the computational complexity for computing the gradient $\nabla \Lambda(\lambda_j)$ in each implementation of Algorithm \ref{algo: ssn} includes approximately $4\iota md$ multiplications and $4\iota md + 3m^2d/r$ additions, along with $\iota d$ simple non-linear operations.

Thirdly, the computation of the direction $d_j$ is inherently the most costly. It involves forming the diagonal matrix $D_j$ and solving the linear system with the coefficient matrix $B_\alpha^{-1} + \mathcal{D}_j$. The estimation of its computational cost is similar to that of $z_j $ and $x_j $, with the difference being that $U^\top C_{\alpha,j} U$ needs to be reassembled at each SSN iteration. Then, it is easy to show that computing $d_j$ requires approximately $4\iota md$ multiplications and approximately $4\iota md + 3\iota  m^2 d$ additions, along with $\iota d$ simple non-linear operations.

Finally,  we need to perform step size searches  with (\ref{SSN_LS}). Due to the potentially varying number of step size searches in each SSN iteration, it's difficult to quantify the computational complexity. However, we note that it is relatively low. In fact, the terms $d_j^\top B_\alpha^{-1}(\lambda_j-g)=d_j^\top x_j$ and   $d_j^\top B_{\alpha}^{-1} d_j$ in \eqref{SSN_LS} can be computed once in each SSN iteration and then reused during the step size searches. Therefore, most of the computational cost of the step size search comes from executing the proximal operator and computing its generalized Jacobian, which consist of $O(\iota d)$ simple non-linear operations, which incurs small computational overhead under Assumption \ref{ass: assumption2}.

In conclusion, each implementation of Algorithm \ref{algo: ssn} requires $O(\iota md)$  multiplications and $O(\iota m^2d)$ additions, along with $O(\iota d)$ simple non-linear operations. We summarize our discussions on computational complexity in Table \ref{table: complexity}.

\begin{table}
	\small
	\centering
	\begin{tabular}{|c|c|c|c|c|}
		\hline
		& Update auxiliary matrices & Compute $\nabla \Lambda(\lambda_j)$ & Compute $d_j$ & Total \\
		\hline
		Multiplications & $4md/r$ & $4\iota md$ & $4\iota md$ & $O(\iota md)$ \\
		\hline
		Additions & - & $4\iota md + 3m^2d/r$ & $4\iota md + 3\iota  m^2 d$ & $O(\iota  m^2 d)$ \\
		\hline
		Non-linear Opts & - & $\iota d$ & $\iota d$ & $O(\iota d)$ \\
		\hline
	\end{tabular}
	\caption{Summary of computational complexity of Algorithm \ref{algo: ssn}.}
	\label{table: complexity}
\end{table}{}

\section{Experiments}
\label{sec: evaluation}

In this section, we present a comprehensive set of experiments to validate the performance, accuracy, and advantages offered by Algorithm \ref{algo: spbfgs} and Algorithm \ref{algo: ssn}. Following the experimental settings of some stochastic Newton-type methods (e.g., \cite{byrd2016stochastic}), we focus on the logistic regression for binary classification, incorporating the elastic net regularizer \cite{zou2005regularization}, which reads
\begin{equation}\label{eq: logistic_prob}
    \min_{x\in\rbb^d}F(x)=\frac{1}{n}\sum_{i=1}^n f_i(x) + h(x),
\end{equation}
where, for all $x\in\rbb^d$,
\begin{equation}\label{eq: logistic_ori}
	\begin{aligned}
		f_i(x) &:= f\left(x ; a_i, b_i\right)=b_i \log \left(c\left(x ; a_i\right)\right)+\left(1-b_i\right) \log \left(1-c\left(x ; a_i\right)\right),\\
		h(x) &:= \frac{\mu}{2} \|x\|^2 + \lambda \|x\|_1.
	\end{aligned}
\end{equation}
In \eqref{eq: logistic_ori}, $(a_i, b_i)\in \rbb^d\times \{0, 1\}$ represents the $i$-th sample, where $a_i\in\rbb^d$ corresponds to the feature vector and $b_i\in\{0,1\}$ represents the label.
The function $c(\cdot; \cdot): \rbb^d\times \rbb^d\rightarrow \rbb$ is defined as:
\begin{equation*}
	c\left(x ; a_i\right)=\frac{1}{1+\exp \left(-a_i^\top x\right)},~\forall x\in\rbb^d.
\end{equation*}
It is worth noting that the square term $\frac{\mu}{2} \|\cdot\|^2$ in the function $h$ is smooth and can be combined with $f_i$, which leads to the following equivalent formulation for all $x\in\rbb^d$:
\begin{equation}\label{eq: logistic}
	\begin{aligned}
		f_i(x) &:= f\left(x ; a_i, b_i\right)=b_i \log \left(c\left(x ; a_i\right)\right)+\left(1-b_i\right) \log \left(1-c\left(x ; a_i\right)\right) + \frac{\mu}{2} \|x\|^2,\\
		h(x) &:= \lambda \|x\|_1.
	\end{aligned}
\end{equation}
An important advantage of the formulation \eqref{eq: logistic} is that each $f_i$ becomes $\mu$-strongly convex and $(\|a_i\|^2 + \mu)$-smooth.

In our experimental evaluation, we compare the performance of Algorithm \ref{algo: spbfgs} with other three algorithms for solving the regularized logistic regression problem \eqref{eq: logistic_prob}, namely the stochastic proximal quasi-Newton method combined with the SVRG (SPQN-SVRG) (e.g., \cite{luo2016proximal}), the stochastic proximal quasi-Newton (SPQN) built upon \cite{byrd2016stochastic}, and Algorithm \ref{algo: plsvrg} (P-LSVRG).
For the step size selection, we follow \cite{byrd2016stochastic} to gradually decrease the step sizes of SPQN. For the other three algorithms, we use a constant step size, chosen for each algorithm from the set $\{2^{-20}, \dots, 2^{-1}, 1\}$ that gives the best overall performance.
We summarize the test algorithms and their step size choices in Table \ref{table: tested_algo}.

In all the experiments, we set the regularization parameter $\mu = \lambda = 10^{-3}$.
The stochastic gradient of each algorithm is computed using a small batch of samples.
In addition, since the subproblem \eqref{eq:subproblem} has no closed-form solution, we solve it iteratively by the ISTA \cite{daubechies2004iterative}, the FISTA \cite{beck2009fast}, or Algorithm \ref{algo: ssn}.
\begin{table}
	\centering
	\begin{tabular}{|c|c|c|c|c|}
		\hline
		Algorithm & Algorithm \ref{algo: spbfgs} & SPQN-SVRG & SPQN & P-LSVRG \\
		\hline
		Step size choice & constant & constant & $\eta_k = O(\frac{1}{k})$ & constant \\
		\hline
	\end{tabular}
	\caption{Summary of algorithms and their step size choices.}
	\label{table: tested_algo}
\end{table}{}

\begin{table}
	\centering
	\begin{tabular}{|c|c|c|c|}
		\hline
		Dataset & $n$ & $d$ & \textit{Sparsity} \\
		\hline $Synthetic1$ & $10^4$&5000&dense \\
		\hline
		$Synthetic2$ & $10^4$&$10^6$&sparsity=$0.1\%$ \\
		\hline
		$Synthetic3$ & $10^4$&$10^6$&sparsity=$1\%$ \\
		\hline $rcv1$ & 23149 & 47236 &sparse \\
		\hline $a9a$ & 32561 & 123 & sparse \\
		\hline $w8a$ & 49749 & 300 & sparse \\
		\hline $mushrooms$ & 8124 & 112 & sparse \\
		\hline
	\end{tabular}
	\caption{Summary of datasets.}
	\label{table: real_dataset}
\end{table}{}

Our experiments for solving \eqref{eq: logistic_prob} consist of the following three scenarios:
\begin{itemize}
	\item Comparison among Algorithm \ref{algo: spbfgs}, SPQN-SVRG, SPQN, and P-LSVRG, as well as inner solvers Algorithm \ref{algo: ssn}, FISTA, and ISTA on synthetic datasets with various dimensions and levels of sparsity.
	\item Comparison between Algorithm \ref{algo: spbfgs} and SPQN-SVRG with different parameter settings on real datasets.
	\item Comparison among testing errors and prediction accuracy of Algorithm \ref{algo: spbfgs}, SPQN, and P-LSVRG on real datasets.
\end{itemize}
To perform the experiments, we use both synthetic and real datasets given below (see also Table \ref{table: real_dataset}).
\begin{itemize}
	\item Three synthetic datasets consisting of 10,000 training samples, where the features are generated from a standard multivariate Gaussian distribution. These datasets include a dense one ($d=5000$), and two sparse ones ($d=10^6$) with sparsity levels of $0.1\%$ and $1\%$ respectively.
	\item Four sparse real datasets including the $rcv1$ dataset \cite{lewis2004rcv1}, as well as three datasets sourced from the LIBSVM library \cite{chang2008libsvm}: $a9a$, $w8a$, and $mushrooms$.
\end{itemize}

In all the figures, we define the \textit{training error} as $F(x) - F^*$, where $F(x)$ is determined based on the data points from the training set, and $F^*$ is obtained by running FISTA for a large number of iterations. On the other hand, the \textit{testing error} is defined as $F(x)$, excluding the regularization term, and using the data points from the testing set.
Our codes are available at 
\url{https://github.com/wzm0213/single-loop-stochastic-proximal-L-BFGS-code.git}, written in Python with \textit{PyTorch}, and all the numerical experiments were conducted on a Windows 11 system, using a laptop equipped with an Intel(R) Core(TM) i5-12500H CPU operating at 2.50 GHz and 16 GB of memory. 

\subsection{Experiments on synthetic datasets}
In this section, we perform experiments on the comparison among Algorithm \ref{algo: spbfgs}, SPQN-SVRG, SPQN, and P-LSVRG, as well as the inner solvers Algorithm \ref{algo: ssn}, FISTA, and ISTA on the synthetic datasets presented in Table \ref{table: real_dataset}.

\subsubsection{Performance of all four algorithms}\label{subsubsec: all-four}
We train the logistic regression model \eqref{eq: logistic_prob} using the four algorithms listed in Table \ref{table: tested_algo} with the three synthetic datasets presented in Table \ref{table: real_dataset}.
The subproblems involved in Algorithm \ref{algo: spbfgs}, SPQN-SVRG and SPQN are solved by FISTA.
Each subproblem takes the form of \eqref{eq: sub_sim} and its first-order optimality condition suggests that its solution is a root of
\begin{equation*}
	\mathcal{E}(x):=x-\p_{\theta}\left(x-B x-g\right).
\end{equation*}
We terminate the inner solver and return $\tilde{x}$ whenever $\mathcal{E}(\tilde{x}) < 10^{-8}$.
For all the three datasets, the batch size for computing the stochastic gradients is chosen to be $b=128$, while the sample size $b_H:=|\mathcal{S}^t|$ for implementing the Hessian-vector product \eqref{eq: hess-vec} is fixed as $600$ for each algorithm.
We also fix the Hessian update frequency $r=10$ for updating the correction pairs, and the maximum memory size $l=10$ for Algorithm \ref{algo: spbfgs}, SPQN-SVRG, and SPQN for a fair comparison.
Moreover, the probability parameter in Algorithm \ref{algo: spbfgs} is fixed as $p=\frac{b}{n}$, and the number of inner iterations in SPQN-SVRG is chosen as $l_s = \frac{1}{p} = \frac{n}{b}$.
The initial points are set as $0.01\cdot \mathbf{1}$ for all the algorithms, with $\mathbf{1}\in \mathbb{R}^d$ being the all-ones vector, while the step sizes are carefully tuned for each algorithm on each dataset.

\begin{figure}
	\centering
	\includegraphics[width=1.0\textwidth]{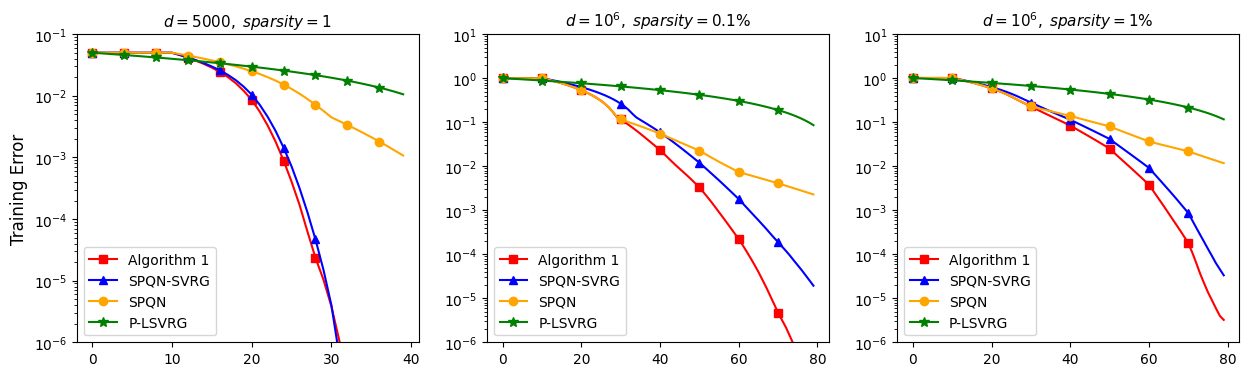}
	\caption{Comparisons for training logistic regression model with three synthetic datasets.}
	\label{fig: synthetic}
\end{figure}

Figure \ref{fig: synthetic} displays the training errors on the three synthetic datasets.
We can observe that Algorithm \ref{algo: spbfgs} and SPQN-SVRG significantly outperform SPQN and P-LSVRG across all three datasets, indicating that integrating VR techniques and Hessian information can expedite algorithmic convergence.

\subsubsection{Comparison among inner solvers}
\label{sec: exp_inner_solver}
In this subsection, we compare the efficiency of Algorithm \ref{algo: ssn} with ISTA and FISTA.
For this purpose, we apply Algorithm \ref{algo: spbfgs} with the above-mentioned three inner solvers to solve \eqref{eq: logistic_prob} on the three synthetic datasets.
We terminate Algorithm \ref{algo: spbfgs} when the relative error $\left[F(x) - F^*\right] /F^* < 10^{-6}$.
The stopping criterion for inner iterations is set as $\mathcal{E}(x) < 10^{-8}$, the initial point is set as $0.01\cdot \mathbf{1}$ for all inner solvers, while the other settings of Algorithm \ref{algo: spbfgs} are kept identical for all the cases to ensure fair comparisons. The performance of the three inner solvers is presented in Table \ref{table: inner_solver}, and we can conclude that Algorithm \ref{algo: ssn} demonstrates superior performance compared to FISTA and ISTA across various problem sizes, as evidenced by significantly better results in all aspects.
Consequently, Algorithm \ref{algo: ssn} has the potential to be employed as a subroutine in diverse proximal Newton-type methods, enabling them to handle high-dimensional problems.

\begin{table}
	\centering
	\begin{tabular}{|c|c|c|c|}
		\hline {\small \textbf{Dataset}} & {\small \textbf{Ave Time (s)}} & {\small \textbf{Ave Iter}} & {\small \textbf{Max Iter}}\\
		\hline \multicolumn{4}{|c|}{Algorithm \ref{algo: ssn}}\\
		\hline $Synthetic1$ & 0.039 & 7.61 & 19 \\
		\hline $Synthetic2$ & 0.125 & 8.26 & 23 \\
		\hline $Synthetic3$ & 0.122 & 8.07 & 23 \\
		\hline \multicolumn{4}{|c|}{FISTA}\\
		\hline $Synthetic1$ & 0.315 & 113.46 & 304 \\
		\hline $Synthetic2$  & 1.877 & 132.51 & 381 \\
		\hline $Synthetic3$ & 1.901 & 137.82 & 452 \\
		\hline \multicolumn{4}{|c|}{ISTA}\\
		\hline $Synthetic1$ & 0.481 & 187.95 & 491 \\
		\hline $Synthetic2$  & 2.753 & 201.68 & 606 \\
		\hline $Synthetic3$ & 2.686 & 197.67 & 622 \\
		\hline
	\end{tabular}
	\caption{Comparison among three inner solvers when implementing Algorithm \ref{algo: spbfgs} on three synthetic datasets. Ave Time: average time for solving the subproblems; Ave Iter: average number of inner iterations; Max Iter: maximum number of inner iterations.}
	\label{table: inner_solver}
\end{table}

\subsection{Comparison between Algorithm \ref{algo: spbfgs} and SPQN-SVRG}
The L-SVRG \cite{kovalev2020don} presents experimental evidence supporting the claim that L-SVRG exhibits faster convergence compared to the SVRG across various tasks.
In our study, we extend this comparison by evaluating the performance of Algorithm \ref{algo: spbfgs} in comparison to SPQN-SVRG. To this end, we use four real datasets: $rcv1$, $a9a$, $w8a$, and $mushrooms$, see Table \ref{table: real_dataset}.

We implement Algorithm \ref{algo: spbfgs} and SPQN-SVRG with different parameters choices for $p$ and $l_s$.
Recall that in this context, $p$ represents the probability of updating the reference point in Algorithm \ref{algo: spbfgs} (see \eqref{eq:update_w}), while $l_s$ denotes the number of inner iterations in SPQN-SVRG.
For fair comparisons, we use Algorithm \ref{algo: ssn} as the inner solver for both algorithms across all datasets, given its high efficiency demonstrated in Section \ref{sec: exp_inner_solver}.
Additionally, we maintain consistent batch sizes of $b = 128$ and $b_H =600$ across all cases.
As it is expensive to estimate the bounds $m_1$ and $m_2$ for the Hessian approximations $B_k$ as in Lemma \ref{lemma: bounded hessian}, we follow Corollary \ref{coro: 1} to select three distinct values of $p$: $\frac{b}{n}$, $\frac{\mu b}{L}$, and $\frac{1}{2}(\frac{b}{n} + \frac{\mu b}{L})$.
We set $l_s = \frac{1}{p}$ to ensure that the expected number of iterations for updating the reference point remain the same for both algorithms.
Furthermore, all the other parameters are the same as those specified in Section \ref{subsubsec: all-four} for both Algorithm \ref{algo: spbfgs} and SPQN-SVRG.

\begin{figure}
	\centering
	\begin{subfigure}[!h]{0.47\linewidth}
		\centering
		\includegraphics[width=\linewidth]{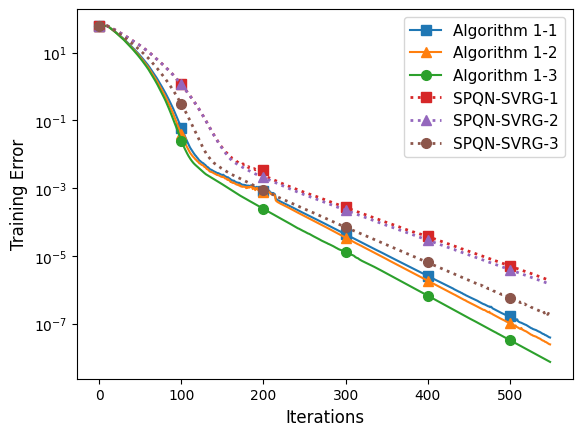}
		\caption{$rcv1$ dataset}
		\label{fig: rcv_comp}
	\end{subfigure}
	\hfill
	\begin{subfigure}[!h]{0.47\linewidth}
		\centering
		\includegraphics[width=\linewidth]{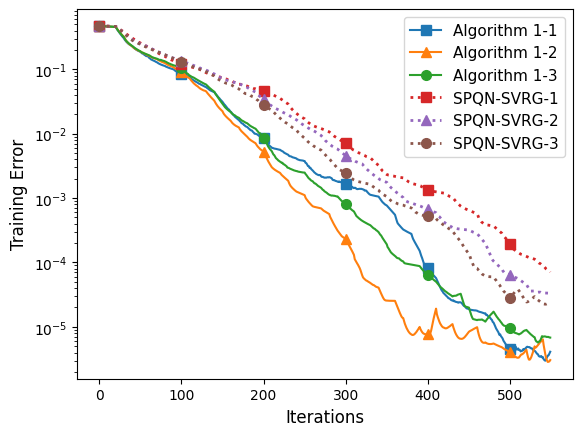}
		\caption{$a9a$ dataset}
		\label{fig: a9a_comp}
	\end{subfigure}
	\hfill
	\begin{subfigure}[!h]{0.47\linewidth}
		\centering
		\includegraphics[width=\linewidth]{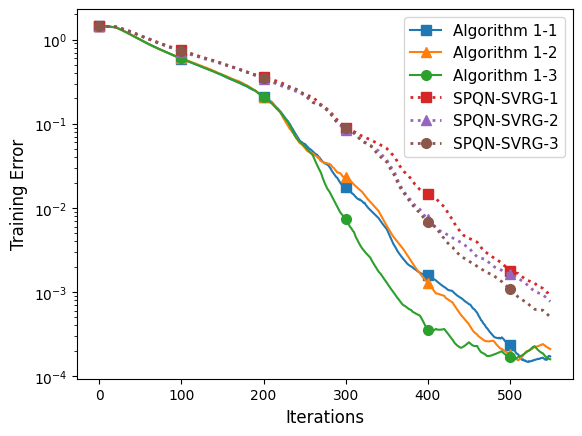}
		\caption{$w8a$ dataset}
		\label{fig: w8a_comp}
	\end{subfigure}
	\hfill
	\begin{subfigure}[!h]{0.47\linewidth}
		\centering
		\includegraphics[width=\linewidth]{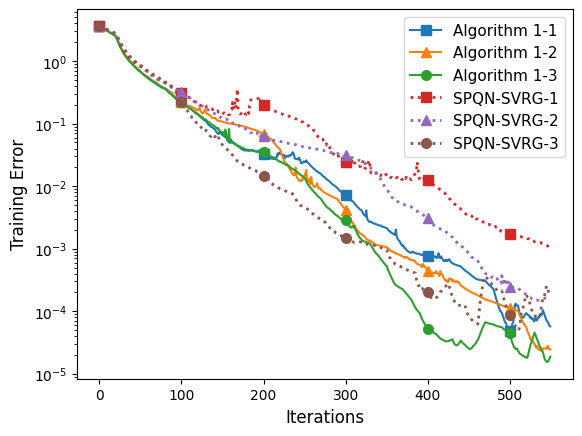}
		\caption{$mushrooms$ dataset}
		\label{fig: mushroom_comp}
	\end{subfigure}
	\caption{Comparison between Algorithm \ref{algo: spbfgs} and SPQN-SVRG over four real datasets using different choices of parameters $p$ and $l_s=1/p$, where the indexes \textbf{-1, -2, -3} correspond to the cases where $p=\frac{b}{n}$, $\frac{1}{2}(\frac{b}{n} + \frac{\mu b}{L})$, and $\frac{\mu b}{L}$ respectively.}
	\label{fig: comp_LSVRG_SVRG}
\end{figure}

Figure \ref{fig: comp_LSVRG_SVRG} reports the resulting training errors on the four real datasets.
The figures effectively demonstrate that the training errors obtained by Algorithm \ref{algo: spbfgs} exhibit faster decreases compared to those obtained by SPQN-SVRG across various parameters choices, and therefore highlight the superior performance of Algorithm \ref{algo: spbfgs} for these tasks.

\subsection{Testing error and prediction accuracy}
In our previous experiments, we observed a rapid decrease in the training errors when using Algorithm \ref{algo: spbfgs}. Moreover, it consistently outperformed SPQN-SVRG across various parameters choices.
Nevertheless, it is essential to evaluate the testing error of the algorithm to assess its performance in avoiding overfitting.
For this purpose, we use two real datasets $a9a$ and $w8a$, which have been split into the training set and testing set beforehand.
For each dataset, we train the logistic regression model \eqref{eq: main prob} using three different algorithms: Algorithm \ref{algo: spbfgs}, SPQN, and P-LSVRG.
The purpose of comparing these three methods is to highlight the advantages gained through the integration of Hessian information and VR techniques in the design of stochastic algorithms.

In the experiments, the parameters $\mu$, $\lambda$, $b$, and $b_H$ are kept consistent with the values specified in Section \ref{subsubsec: all-four}. 
The subproblems involved in Algorithm \ref{algo: spbfgs} and SPQN are solved by Algorithm \ref{algo: ssn}.
We record the iterates $\{x_k\}$ along the training process to calculate the corresponding testing errors on the testing dataset for each algorithm.

\begin{figure}
	\centering
	\begin{subfigure}[!h]{0.48\linewidth}
		\centering
		\includegraphics[width=\linewidth]{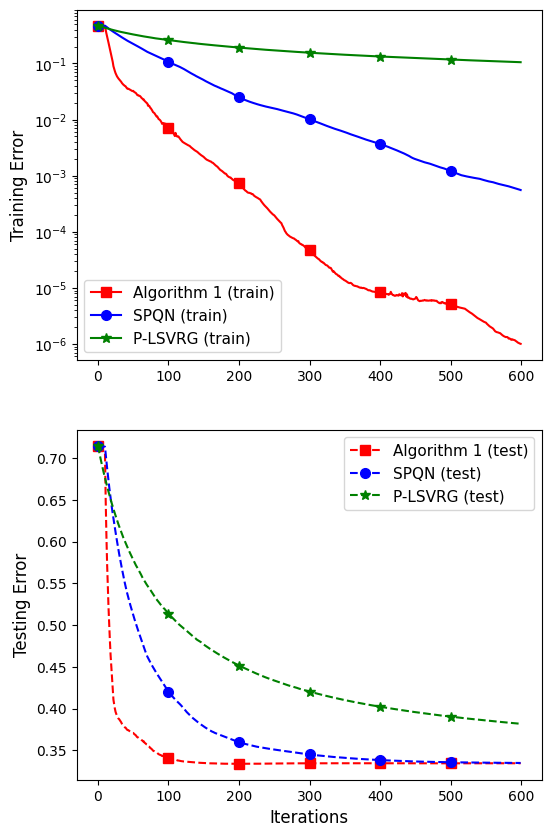}
		\caption{$a9a$ dataset}
		\label{fig: a9a_general}
	\end{subfigure}
	\hfill
	\begin{subfigure}[!h]{0.48\linewidth}
		\centering
		\includegraphics[width=\linewidth]{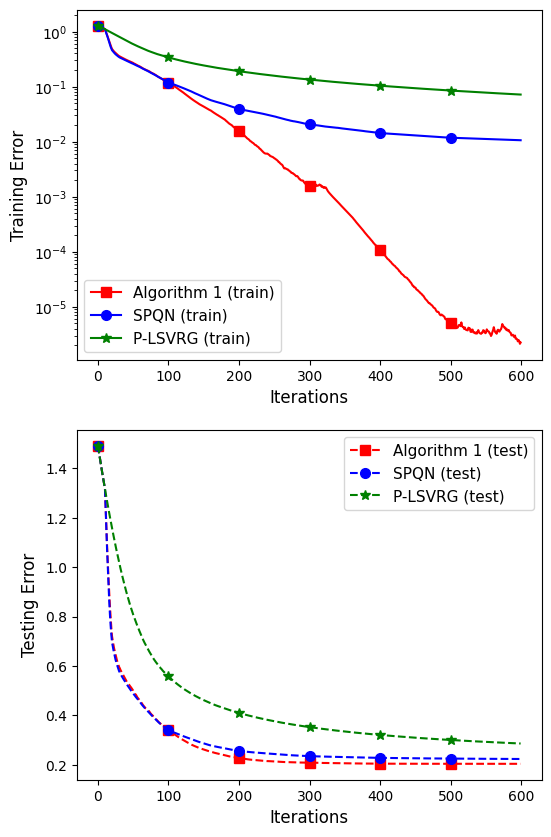}
		\caption{$w8a$ dataset}
		\label{fig: w8a_general}
	\end{subfigure}
	\caption{Training and testing error for the logistic regression model on $a9a$ and $w8a$ datasets trained with three algorithms: Algorithm \ref{algo: spbfgs}, SPQN, and P-LSVRG.}
	\label{fig: generalization_error}
\end{figure}

Figure \ref{fig: generalization_error} displays the training and testing errors for each algorithm. Based on the figures, we observe that Algorithm \ref{algo: spbfgs} consistently outperforms the other two algorithms in terms of both training and testing errors.
This observation supports our expectation that the integration of Hessian information and VR techniques in Algorithm \ref{algo: spbfgs} leads to significant improvements in training efficiency while effectively avoiding overfitting.

\begin{figure}
	\centering
	\begin{subfigure}[h]{0.47\linewidth}
		\centering
		\includegraphics[width=\linewidth]{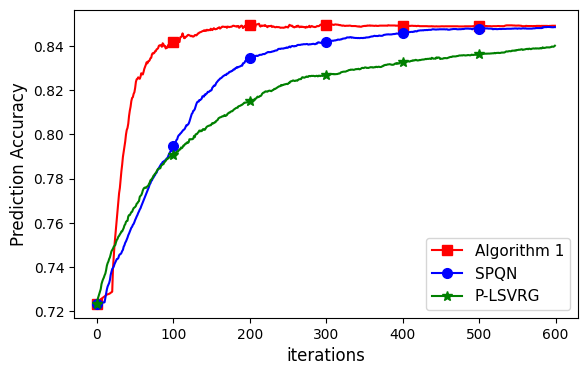}
		\caption{$a9a$ dataset}
		\label{fig: a9a_accuracy}
	\end{subfigure}
	\hfill
	\begin{subfigure}[h]{0.47\linewidth}
		\centering
		\includegraphics[width=\linewidth]{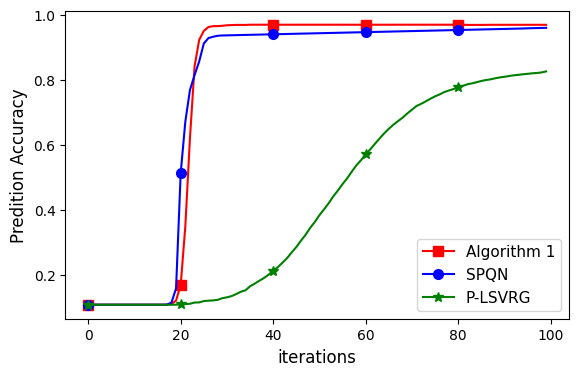}
		\caption{$w8a$ dataset}
		\label{fig: w8a_accuracy}
	\end{subfigure}
	\caption{Prediction accuracy for the logistic regression model on $a9a$ and $w8a$ datasets trained with three algorithms: Algorithm \ref{algo: spbfgs}, SPQN, and P-LSVRG.}
	\label{fig: test_accuracy}
\end{figure}

We also present the prediction accuracy in Figure \ref{fig: test_accuracy} to provide additional insights into the performance of these algorithms.
From Figure \ref{fig: test_accuracy}, it can be observed that while both Algorithm \ref{algo: spbfgs} and SPQN achieve high prediction accuracy within a small number of iterations, Algorithm \ref{algo: spbfgs} demonstrates relatively better performance in terms of prediction accuracy compared to SPQN.
In contrast, the prediction accuracy of the P-LSVRG increases relatively slowly.
This observation further supports the superior performance of Algorithm \ref{algo: spbfgs} in terms of efficient convergence and achieving high prediction accuracy on the logistic binary classification task.

\section{Conclusions}
\label{sec: conclusion}
In this paper, we present a novel stochastic proximal quasi-Newton method for efficiently solving a class of convex optimization problems commonly encountered in machine learning, which involve structured objective functions consist of both smooth and nonsmooth components. The proposed method offers several advantages, which we summarize as follows:
\begin{itemize}
	\item Structurally, our method combines the loopless SVRG technique with a stochastic L-BFGS scheme. By incorporating Hessian information and leveraging stochastic gradients with reduced variances, our method produces high-quality search directions in large-scale scenarios. Furthermore, compared to the commonly used SVRG scheme, the generation of stochastic gradients via the loopless SVRG only exhibits a single loop, simplifying the implementation of our method.
	\item Theoretically, we establish a worst-case globally linear convergence rate for the proposed method under mild assumptions. We also discuss and demonstrate linear convergence for special cases and generalizations of our method, which allows for incorporating other variance reduction techniques.
	\item Numerically, we propose a fast Semismooth Newton (SSN) method along with a line search scheme to efficiently solve the subproblems. By leveraging a compact representation of the L-BFGS matrix and storing a few auxiliary matrices, we significantly reduce the computational burdens for implementing our method.
\end{itemize}

Both synthetic and real datasets are tested on a regularized logistic regression problem to validate the computational efficiency of our method along with the SSN solver. The results suggest that the proposed method outperforms several state-of-the-art algorithms, and the SSN-based numerical implementation significantly reduces the overall computational cost.

\section*{Acknowledgments}
The authors are grateful to two anonymous referees for their very valuable comments which have helped improve the paper substantially.
The work of Y. Song was supported by a Start-Up Grant from Nanyang Technological University.
The work of X. Yuan was supported by the Hong Kong Research Grants Council under the GRF project 17309824.
The work of H. Yue was supported by the National Natural Science Foundation of China (No. 12301399).

{\small
\bibliographystyle{ams}

}

\end{document}